\documentclass{amsart}
\usepackage{dsfont}
\usepackage{mathrsfs}
\usepackage{color}
\usepackage{CJK}
\usepackage{amssymb}
\newtheorem{theorem}{Theorem}[section]
\newtheorem{lemma}[theorem]{Lemma}
\usepackage[nosort]{cite}
\theoremstyle{definition}
\newtheorem{definition}{Definition}[section]
\newtheorem{example}[theorem]{Example}

\newtheorem{proposition}{Proposition}[section]
\newtheorem{corollary}[theorem]{Corollary}

\newtheorem{condition}[theorem]{Condition}
\theoremstyle{remark}

\newtheorem{remark}{Remark}[section]

\numberwithin{equation}{section}

\makeatletter

\newcommand{\Rmnum}[1]{\expandafter\@slowromancap\romannumeral #1@}
\makeatother

\begin{document}

\title{ Topological pressure of free semigroup actions for non-compact sets and Bowen's equation}


\author{Qian Xiao}
\address{School of Mathematics, South China University of Technology, Guangzhou 510641, P.R. China}
\email{qianqian1309200581@163.com}

\author{Dongkui Ma*}
\thanks{* Corresponding author}
\address{School of Mathematics, South China University of Technology,
Guangzhou 510641, P.R. China}
\email{dkma@scut.edu.cn}





\begin{abstract}
Climenhaga [V. Climenhaga, Bowen's equation in the non-uniform setting, Ergodic Theory Dynam. Systems (2011), 31, 1163-1182] showed the applicability of Bowen's equation to arbitrary subset of a compact metric space. The main purpose of this paper is to generalize the main result of Climenhaga to free semigroup actions for non-compact sets. We introduce the notions of the topological pressure and lower and upper capacity topological pressure of a free semigroup action for non-compact sets by using the Carath\'{e}odory-Pesin structure (C-P structure). Some properties of these notions are given, followed by three main results. One is to characterize the Hausdorff dimension of arbitrary subset in term of the topological pressure by Bowen's equation, whose points have the positive lower Lyapunov exponents relative to $\omega \in \Sigma_m^+$ and satisfy a tempered contraction condition, the other is the estimation of topological pressure of a free semigroup action on arbitrary subset of $X$ and the third is the relationship between the upper capacity topological pressure of a skew-product transformation and the upper capacity topological pressure of a free semigroup action with respect to arbitrary subset.
\end{abstract}

\maketitle

\section{ \emph{Introduction}}

As an application of topological pressure, Bowen's equation has an important impact on dynamical system. Bowen\cite{Bowen3} first gave the connection between topological pressure and Hausdorff dimension and proved that for a certain compact sets (quasi-circles) $J \in \mathbb{C}$ which arise as invariant sets of fractional linear transformations $f$ of the Riemann sphere, the Hausdorff dimension $t=dim_{H} J$ is the unique root of the equation
\begin{equation}
P_{J}(-t\varphi)=0,
\end{equation}
where $P_{J}$ is the topological pressure of the map $f: J \rightarrow J$, and $\varphi$ is the geometric potential $\varphi(Z)=\log |f'(z)|$. Later, Ruelle\cite{Ruelle} showed that Bowen's equation (1.1) gives the Hausdorff dimension of $J$ whenever $f$ is a $C^{1+\varepsilon}$ conformal map on a Riemann manifold and $J$ is a repeller. This result was eventually extended to the case where $f$ is $C^{1}$ by Gatzouras and Peres\cite{Gatzouras}. Later, one can also give a defintion of conformal map in the case where $X$ is a metric space (not necessarily a manifold), and the analogous result was proved by Rugh\cite{Rugh}. For the so-called parabolic Cantor set and Julia set, the authers charaterized Hausdorff dimension with Bowen's equation (see, for example,\cite{Ban,Cao,Denker,Mayer,Mayer1,Przytycki,Przytycki1,Urbanski1,Urbanski2,Urbanski3}).

Given a map $f$, all of the above results give the Hausdorff dimension of one very particular dynamically significant set $J$ via Bowen's equation. It is natural to ask if one can give the Hausdorff  dimension of subset $Z \subset J$ via a similar approach.

For certain subsets, results in this direction are given by the multifractal analysis (see, for example, \cite{Ba1,Pesin2,Weiss}).

Pesin \cite{Pesin}  gave a new characterization of topological pressure for non-compact sets by Carath\'{e}odory structure, which we call  Carath\'{e}odory-Pesin structure or C-P structure for short. This extended the earlier definition of topological entropy for non-compact sets by Bowen\cite{Bowen}.
Using the C-P structure introduced in \cite{Pesin}, Barreira and Schmeling \cite{Ba2} introduced the notion of the $\emph{u-dimension}$ $dim_{u} Z$ for positive functions $u$, showing that $dim_{u} Z$ is the unique number $t$ such that $P_{Z}(-tu)=0$. They also showed that for a subset $Z$ of a conformal repeller $J$, where we may take $u=\log \|Df\|>0$, we have $dim_{u} Z=dim_{H} Z$, hence the Hausdorff dimension of any subset $Z \subset J$ is given by Bowen's equation, whether or not $Z$ is compact or invariant.
 Moreover, Climenhaga\cite{Climenhaga} showed that the applicability of Bowen's equation to arbitrary $Z$, requiring only that the lower Lyapunov exponents be positive on $Z$, together with a tempered contraction condition, which extends beyond the uniformly expanding case.
\begin{definition}\label{7.1d}
$f: X \longrightarrow X$ is called \emph{conformal} with factor $a(x)$ if for every $x \in X$, we have
\[
a(x)=\lim_{y\rightarrow x}\frac{d(f(x),f(y))}{d(x,y)},
\]
where $a:X\longrightarrow [0,\infty)$ is continuous.
\end{definition}
Denote the Birkhoff sums of $\log a$ by
\[
\lambda_{n}(x)=\frac{1}{n}S_{n}(\log a)(x)=\frac{1}{n}\sum_{k=0}^{n-1} \log a(f^{k}(x));
\]
the lower and upper limits of this sequence are the \emph{lower Lyapunov exponent} and \emph{upper Lyapunov exponent}, respectively:
\[
\underline{\lambda}(x)=\liminf_{n\rightarrow \infty}\lambda_{n}(x),~~~~\overline{\lambda}(x)=\limsup_{n\rightarrow \infty}\lambda_{n}(x).
\]
If the limit exists, then their common value is the \emph{Lyapunov exponent}:
\[
\lambda(x)=\lim_{n\rightarrow \infty}\lambda_{n}(x).
\]
In studying the relationship between the Hausdorff dimension of $Z$ and the topological pressure of $\log a$ on $Z$, provided every point in $Z$ has positive lower Lyapunov exponent and satisfies the following \emph{tempered contraction condition} (\cite{Climenhaga}):
\begin{equation}\label{1.2}
\inf_{n \in \mathbb{N},~0\leq k \leq n} \{S_{n-k} \log a(f^{k}(x))+n\varepsilon\} > -\infty ~~for~every~\varepsilon>0.
\end{equation}
Denote $\mathcal{B}$ by the set of all points in $X$ which satisfy (\ref{1.2}).
Given $E \subset \mathbb{R}$ and let $\mathcal{A}(E)$ be the set of points along whose orbits all the asymptotic exponential expansion rates of the map $f$ lie in $E$:
\[
\mathcal{A}(E)=\{x\in X:[\underline{\lambda}(x),\overline{\lambda}(x)]\subseteq E\}.
\]
And $\mathcal{A}(\alpha)=\mathcal{A}(\{\alpha\})$.

More precisely, Climenhaga\cite{Climenhaga} proved the following ([13, Theorem 2.4]) :
\begin{theorem}
Let $X$ be a compact metric space and $f: X\rightarrow X$ be continuous and conformal with factor $a(x)$. Suppose that $f$
has no critical points and no singularities-that is, that $0<a(x)<\infty$ for all $x \in X$. Consider $Z \subset \mathcal{A}((0,\infty))\bigcap \mathcal{B}$. Then the Hausdorff dimension of $Z$ is given by
\begin{align*}
dim_{H} Z=t^{*}&=\sup\{ t\geq 0: P_{Z}(f,-t \log a )>0\}\\
&=\inf \{ t\geq 0: P_{Z}(f,-t \log a ) \leq 0\}.
\end{align*}
Furthermore, if $Z \subset \mathcal{A}((\alpha,\infty))\bigcap \mathcal{B}$ for some $\alpha >0$, then $t^{*}$ is the unique root of Bowen's equation
\[
P_{Z}(f,-t \log a )=0.
\]
Finally, if $Z \subset \mathcal{A}(\alpha)$ for some $\alpha >0$, then $P_{Z}(f,-t \log a)=h_{top}(Z)-t\alpha$, and hence
\[
dim_{H} Z=\frac{1}{\alpha} h_{top}(Z).
\]
where $P_{Z}(f,-t \log a)$ denotes the topological pressure of $f$ with respect to $-t \log a$ on $Z$, $h_{top}(Z)$ is the topological entropy on $Z$ and $dim_{H}(Z)$ is the Hausdorff dimension of $Z$.
\end{theorem}

The above result is for a single map. In the past two decades, people have studied a series of free semigroup actions. Related studies include  \cite{Bis,Bufetov,Carvalho1,Carvalho2,Carvalho3,Ju,Lin,Ma,Rodrigues,Wang,Wang1}. Naturally, we wonder if the result of Theorem 1.1 is true in the case of free semigroup actions. Hence in this paper, we introduce the notions of new topological pressure by using C-P structure. By using this topological pressure, we can answer the above question and then generalize the result of Climenhaga \cite{Climenhaga}.

This paper is organized as follows. In section 2, we give our main results. In section 3, we give some preliminaries. In section 4, by using the C-P structure we give the new definitions of the  topological pressure, lower and upper capacity topological pressure of a free semigroup action and several of their properties are provided. In section 5, we give the other two definitions of the topological pressure, the upper and lower capacity topological pressure of a free semigroup action on arbitrary subset of $X$, which is equivalent to the definition in section 4. In section 6, 7, 8, we give the proofs of the main results, respectively.

\section{\emph{Statement of main results}}

Let $(X,d)$ be a compact metric space, denote $G$ the free semigroup with $m$ generators $\{f_0,f_1,\ldots,f_{m-1}\}$ acting on X and for each $f_{i} \in G$, $f_{i}$  is continuous and conformal with factor $a_{i}(x)$. Let $\Phi=\{\log a_{0},\log a_{1},\cdots, \log a_{m-1}\}$. For $w=i_{1}i_{2}\cdots i_{n} \in F_{m}^{+}$, where $F_{m}^{+}$ denotes the set of all finite words of symbols $0,1,\ldots,m-1$, denote
\[
S_{w}\Phi(x):=\log a_{i_{1}}(x)+\log a_{i_{2}}(f_{i_{1}}(x))+\cdots+\log a_{i_{n}}(f_{i_{n-1}i_{n-2}\cdots i_{1}}(x))
\]
and
\begin{align*}
\lambda_{w}(x)&=\frac{1}{|w|}S_{w}\Phi(x),
\end{align*}
where
\[
a_{i}(x)=\lim\limits_{y\rightarrow x} \frac{d(f_{i}(x),f_{i}(y))}{d(x,y)}, ~~i=0,1,\cdots,m-1,
\]
and
\[
f_{i_{n-1}i_{n-2}\cdots i_{1}}:=f_{i_{n-1}}\circ f_{i_{n-2}} \circ \cdots \circ f_{i_{1}}.
\]
And for any $\omega =(i_{0},i_{1},\cdots)\in \Sigma_m^+$, where $\Sigma_m^+$ denotes the one-side symbol space, denote $\omega|_{[0,n-1]}:=i_{0}i_{1}\cdots i_{n-1}$ and set
\[
\underline{\lambda}_{\omega}(x)=\liminf_{n\rightarrow \infty} \lambda_{\omega|_{[0,n-1]}}(x),
\]
and
\[
\overline{\lambda}_{\omega}(x)=\limsup_{n\rightarrow \infty} \lambda_{\omega|_{[0,n-1]}}(x);
\]
then we call $\underline{\lambda}_{\omega}(x)$ and $\overline{\lambda}_{\omega}(x)$ the \emph{lower and upper Lyapunov exponents} of the free semigroup $G$ relative to $\omega \in \Sigma_m^+$ at $x$. If the two are equal (i.e., if the limit exists), then their common value is the \emph{Lyapunov exponent} relative to $\omega \in \Sigma_m^+$ at $x$:
\[
\lambda_{\omega}(x)=\lim_{n\rightarrow \infty} \lambda_{\omega|_{[0,n-1]}}(x).
\]

Our main result relates the Hausdorff dimension of $Z$ to the topological pressure of $\Phi$ on $Z$, provided every point in $Z$ has $\underline{\lambda}_{\omega}(x)>0$ and satisfies the following so-called \emph{tempered contraction condition}:

\begin{equation}\label{7.2}
\inf_{w\in F_{m}^{+}, \overline{w'}\leq \overline{w}}\{ S_{w}\Phi(x)-S_{w'}\Phi(x)+|w|\varepsilon\}>-\infty,~for~any~\varepsilon>0.
\end{equation}
Denote $\mathcal{B}$ as the set of all points in X which satisfy (\ref{7.2}) .

\begin{remark}
 (1) Proposition 6.2 shows that if for any the $\omega \in \Sigma_m^+$, the Lyapunov exponent of $x$ exists and is positive, that is, if $\underline{\lambda}_{\omega}(x)=\overline{\lambda}_{\omega}(x)>0$ for every $\omega \in \Sigma_m^+$, then $x$ satisfies (\ref{7.2}).\\
 (2) We observe that if $a_{i}(x) \geq 1$ for any $x \in X$ and $i=0,1,\cdots,m-1$, then $\mathcal{B}=X$.\\
 (3) We say that $x$ has \emph{bounded  contraction} if $\inf\{ S_{w}\Phi(x)-S_{w'}\Phi(x): w\in F_{m}^{+}, \overline{w'}\leq \overline{w}\}>-\infty.$ Any point which has bounded contraction  satisfies (\ref{7.2}).
\end{remark}

Given $E \subset \mathbb{R}$ and let $\mathcal{A}(E)$ be the set of points along whose orbits all the asymptotic exponential expansion rates of the $G$ relative to $\omega \in \Sigma_m^+$ lie in $E$:
\[
\mathcal{A}(E)=\{x\in X:[\underline{\lambda}_{\omega}(x),\overline{\lambda}_{\omega}(x)]\subseteq E, \omega \in  \Sigma_m^+\}.
\]

In particular, $\mathcal{A}((0,\infty))$ is the set of all points for which $\underline{\lambda}_{\omega}(x)>0$ for every $\omega \in  \Sigma_m^+$ and $\mathcal{A}(\alpha)=\mathcal{A}(\{\alpha\})$. The first main result deals with subsets $Z \subset X$ that lie in both $\mathcal{A}((0,\infty))$ and $\mathcal{B}$.

\begin{theorem}\label{7.3t}
Let $(X,d)$ be a compact metric space, $G$ the free semigroup with $m$ generators $\{f_0,f_1,\ldots,f_{m-1}\}$ acting on X and $f_{i}: X\longrightarrow X$ continuous and conformal with factor $a_{i}(x)$. Assume $f_{i}$ has no critical points and singularities, that is, $0< a_{i}(x)< \infty$ for all $x \in X$ and $i \in \{0,1,2,\cdots,m-1\}$. Consider $Z \subset \mathcal{A}((0,\infty))\bigcap \mathcal{B}$ and $\Phi=\{\log a_{0},\log a_{1},\cdots, \log a_{m-1}\}$. Then the Hausdorff dimension of $Z$ is given by
\begin{align*}
dim_{H} Z=t^{*}&=\sup\{ t\geq 0: P_{Z}(G,-t \Phi )>0\}\\
&=\inf \{ t\geq 0: P_{Z}(G,-t \Phi ) \leq 0\}.
\end{align*}
Furthermore, if $Z \subset \mathcal{A}((\alpha,\infty))\bigcap \mathcal{B}$ for some $\alpha >0$, then $t^{*}$ is the unique root of Bowen's equation
\[
P_{Z}(G,-t \Phi )=0.
\]
Finally, if $Z \subset \mathcal{A}(\alpha)$ for some $\alpha >0$, then $P_{Z}(G,-t \Phi )=h_{Z}(G)-t\alpha$, and hence
\[
dim_{H} Z=\frac{1}{\alpha} h_{Z}(G).
\]
Where $h_{Z}(G)$ is the topological entropy on $Z$ defined by Ju et al in \cite{Ju} and $ P_{Z}(G,-t \Phi )$ denotes the topological pressure of $G$ with respect to $-t\Phi$ on $Z$ (see Section 4).
\end{theorem}

\begin{remark}
When $m=1,$ i.e., $G=\{f\},~\Phi=\{\log a\}$, the result coincides with \cite{Climenhaga}.
\end{remark}

Next, we give our second main result. Considering a compact metric space $(X,d)$, a free semigroup $G=\{f_0, f_1, \ldots, f_{m-1}\}$  acting on $X$, where $f_{i}(i=0,1,\cdots,m-1)$ is continuous transformation of $X$ and $\varphi_{0}, \varphi_{1}, \cdots, \varphi_{m-1}  \in C(X,\mathbb{R})$, where $C(X,\mathbb{R})$ denotes the Banach algebra of real-valued continuous functions of $X$ equipped with the supremum norm. Denote $\Phi=\{\varphi_{0}, \varphi_{1}, \cdots, \varphi_{m-1}\} $. For $w=i_{1}i_{2}\cdots i_{n} \in F_{m}^{+}$, $B_{w}(x,r)$ is the $(w,r)$-Bowen ball at $x$ and denote
\[
S_{w}\Phi(x):=\varphi_{i_{1}}(x)+\varphi_{i_{2}}(f_{i_{1}}(x))+\cdots+\varphi_{ i_{n}}(f_{i_{n-1}i_{n-2}\cdots i_{1}}(x)).
\]
Let $\mu$ be a Borel probability measure on $X$ and $w\in F_m^{+}$. Denote
\begin{equation*}
\underline{P}(G, \Phi, x):=\lim\limits_{r\rightarrow 0}\liminf_{n\rightarrow \infty}-\frac{1}{n}\max_{|w|=n}\{\log\mu \big(B_{w}(x,r) \big)-S_{w}\Phi(x)\},
\end{equation*}
and
\begin{equation*}
\overline{P}(G, \Phi, x):=\lim\limits_{r\rightarrow 0}\liminf_{n\rightarrow \infty}-\frac{1}{n}\min_{|w|=n}\{\log\mu \big(B_{w}(x,r) \big)-S_{w}\Phi(x)\}.
\end{equation*}

Now we give two estimates about topology pressure on $Z \subset X$.
\begin{theorem}\label{2.2t}
Let $\mu$ denote a Borel probability measure on $X$, $Z$ be a Borel subset of $X$ and $ s \in (0,\infty)$.
\begin{itemize}
    \item[(1)]
    ~If~ $\underline{P}(G, \Phi, x)\ge s$~ for~ all ~$x\in Z$ ~and ~$\mu(Z)>0$ ~then~  $P_{Z}(G,\Phi) \ge s$.

    \item[(2)]
    ~If~ $\overline{P}(G, \Phi, x)\le s$~ for ~all ~$x\in Z$  then  ~ $P_{Z}(G,\Phi)\le s$.
\end{itemize}
\end{theorem}

\begin{remark}
When $\varphi_{0}=\varphi_{1}=\cdots=\varphi_{m-1}=0$, it is consistent with the result in \cite{Ju}. Moreover, if $ m=1 $,  then the above theorem coincides with the main results that Ma and Wen proved in \cite{Mj} .
\end{remark}

Finally, the third result describes the relationship between the upper capacity topological pressure of a free semigroup action and the upper capacity topological pressure of a skew-product transformation.
Let $X$ be a compact metric space with metric $d$, suppose a free semigroup with $m$ generators acting on $X$, the generators are  continuous transformations $G=\{ f_0, f_1, \ldots, f_{m-1}\}$ of $X$ and $\varphi \in C(X,\mathbb{R})$. Let $F: \Sigma_m \times X \to \Sigma_m \times X$ be a skew-product transformation and $g: \Sigma_m \times X \to \mathbb{R}$ be defined by the formula $g(\omega, x) =c+\varphi(x)$, where $\Sigma_m$ denotes the two-side symbol space and $c$ is a constant.
$\overline{CP}_{ \Sigma_m \times Z}(F,g)$ denotes the upper capacity topological pressure of $F$ with respect to $g$ on $\Sigma_m \times Z$ (see \cite{Pesin}), $\overline{CP}_{Z}(G,\varphi)$ the upper capacity topological pressure of $G$ with respect to $\varphi$ on $Z$ (see Section 4), $P_{ \Sigma_m \times X}(F,g)$ the topological pressure of $F$ with respect to $g$ on $\Sigma_m \times X$ (see \cite{Walters}), $P_{X}(G,\varphi)$ the topological pressure defined by Lin et al in \cite{Lin},
$h_{\Sigma_m \times X}(F)$ the topological entropy (see \cite{Walters}) and $h_{X}(G)$ the topological entropy of $G=\{ f_{0},f_{1},\cdots,f_{m-1}\}$ (see \cite{Bufetov}). Then we have the following theorem:
\begin{theorem}\label{5.2t}
For any set $ Z \subset X $, we have
\begin{align*}
\overline{CP}_{ \Sigma_m \times Z}(F,g) = \log m + \overline{CP}_{Z}(G,\varphi)+c.
\end{align*}
\end{theorem}

\begin{remark}
(1) If $g(\omega, x) =\varphi(x)\equiv 0$, it is $\overline{Ch}_{ \Sigma_m \times Z}(F) = \log m + \overline{Ch}_{Z}(G)$, which is consistent with the result in \cite{Ju}. When $Z=X$, we have $\overline{CP}_{ \Sigma_m \times X}(F,g)= P_{ \Sigma_m \times X}(F,g),~\overline{CP}_{X}(G,\varphi)= P_{X}(G,\varphi)$, then we can get $P_{ \Sigma_m \times X}(F,g)=\log m + P_{X}(G,\varphi)+c$, which has been proved by Lin et al in \cite{Lin}. Further, if $g(\omega, x) =\varphi(x)\equiv 0$, it is easy to get $h_{\Sigma_m \times X}(F)=\log m +h_{X}(G)$, which has been proved by Bufetov in \cite{Bufetov}.\\
(2) This result also retains for one-side symbol system.
\end{remark}

\section{\emph{Preliminaries}}

\subsection{Carath\'{e}odory-Pesin structure}

Let $X$ and $\mathcal{\mathcal{S}}$ be arbitrary sets and $\mathcal{F}=\{U_{s}: s\in \mathcal{S}\}$  a collection of subsets in $X$. From Pesin \cite{Pesin}, we assume that there exist two functions $\eta,\psi:\mathcal{\mathcal{S}}\rightarrow{\mathbb{R}^{+}}$ satisfying the following conditions:

(1) there exists $s_{0}\in \mathcal{\mathcal{S}}$ such that $U_{s_{0}}=\emptyset;$ if $U_{s}=\emptyset$ then  $\eta(s)=\psi(s)=0$; if $U_{s}\not=\emptyset$ then $\eta(s)>0$ and $\psi(s)>0$;

(2) for any $\delta>0$ one can find $\varepsilon>0$ such that $\eta(s)\le\delta$ for any $s\in \mathcal{\mathcal{S}}$ with $\psi(s)\le\varepsilon$;

(3) for any $\varepsilon>0$ there exists a finite or countable subcollection $\mathcal{G}\subset\mathcal{\mathcal{S}}$ which covers $X$ (i.e., $\cup_{s\in\mathcal{G}}U_{s}\supset X$) and $\psi(\mathcal{G})$:=sup\{$\psi(s):s\in \mathcal{\mathcal{G}}$\} $\le\varepsilon$.

Let $\xi:\mathcal{\mathcal{S}}\rightarrow{\mathbb{R}^{+}}$ be a function, we say that the set $\mathcal{\mathcal{S}}$, collection of subsets $\mathcal{F}$, and the set functions $\xi,\eta,\psi $ satisfying conditions (1), (2) and (3), form the Carath\'{e}odory-Pesin structure or C-P structure $\tau$ on $X$ and  write $\tau=(\mathcal{\mathcal{S}},\mathcal{F},\xi,\eta,\psi)$.

Given a subsetset $Z$ of $X$, $\alpha\in\mathbb{R}$, and $\varepsilon>0$, we define
\[
M(Z,\alpha,\varepsilon):= \inf\limits_{\mathcal{G}}\left\{\sum\limits_{s\in\mathcal{G}}\xi(s)\eta(s)^{\alpha}\right\},
\]
where the infimum is taken over all finite or countable subcollections $\mathcal{G} \subset \mathcal{\mathcal{S}}$ covering $Z$ with  $\psi(\mathcal{G})\le\varepsilon$. By Condition $(3)$ the function $M(Z,\alpha,\varepsilon)$ is correctly defined. It is non-decreasing as $\varepsilon$ decreases. Therefore, the following limit exists:
\[
m(Z,\alpha)=\lim\limits_{\varepsilon\rightarrow 0}M(Z,\alpha,\varepsilon).
\]

It was shown in \cite{Pesin} that there exists a critical value $\alpha_{C}\in[-\infty,\infty ]$ such that
\[m(Z,\alpha)=0, ~\alpha>\alpha_{C},\]
\[m(Z,\alpha)=\infty, ~\alpha<\alpha_{C}.\]
The number $\alpha_{C}$ is called the Carath\'{e}odory-Pesin dimension of the set $Z$.

Now we assume that the following condition holds:

$(3^{'})$ there exists $ \epsilon >0$ such that for any $0< \varepsilon\le\epsilon$ there exists a finite or countable subcollection $\mathcal{G} \subset \mathcal{\mathcal{S}}$ covering $X$ such that $\psi(s)=\varepsilon$ for any $s\in \mathcal{G}$.

Given $\alpha\in\mathbb{R}$ and $\varepsilon>0$, for any subset $Z\subset X$, define
\[
R(Z,\alpha,\varepsilon) = \inf\limits_{\mathcal{G}}\left\{\sum\limits_{s\in\mathcal{G}}\xi(s)\eta(s)^{\alpha} \right \},
\]
where the infimum is taken over all finite or countable subcollections $\mathcal{G} \subset\mathcal{S}$ covering $Z$ such that $\psi(s)=\varepsilon$ for any $ s\in \mathcal{G} $.
Set
\[\underline{r}(Z,\alpha)=\liminf_{\varepsilon\to 0}R(Z,\alpha,\varepsilon), ~\overline{r}(Z,\alpha)=\limsup_{\varepsilon\to 0}R(Z,\alpha,\varepsilon).\]

It was shown in \cite{Pesin} that there exist $\underline{\alpha}_{C},\overline{\alpha}_{C}\in \mathbb{R}$  such that
\[
 \underline{r}(Z,\alpha)=\infty, \alpha<\underline{\alpha}_{C},~ \underline{r}(Z,\alpha)=0, \alpha>\underline{\alpha}_{C},
\]
\[
\overline{r}(Z,\alpha)=\infty, \alpha<\overline{\alpha}_{C}, ~\overline{r}(Z,\alpha)=0, \alpha>\overline{\alpha}_{C}.
\]
The numbers $\underline{\alpha}_{C}$ and $\overline{\alpha}_{C}$  are called the lower and upper Carath\'{e}odory-Pesin capacities of the set $Z$ respectively.

For any $ \varepsilon >0 $ and subset $ Z \subset X $, put
\[
\Lambda(Z,\varepsilon):=\inf\limits_{\mathcal{G}}\left\{\sum\limits_{s\in\mathcal{G}}\xi(s) \right\},
\]
where the infimum is taken over all finite or countable subcollections $\mathcal{G} \subset\mathcal{S}$ covering $Z$ such that $\psi(s)=\varepsilon$ for any $ s\in \mathcal{G} $.

Assume that the function $\eta$ satisfies the following condition:

(4) $\eta(s_{1})=\eta(s_{2})$ for any $s_{1}, s_{2}\in \mathcal{\mathcal{S}}$ for which $\psi(s_{1})=\psi(s_{2}) $.

It was shown in \cite{Pesin} that if the function $\eta$ satisfies condition (4) then for any subset  $Z \subset X$,

\[
\underline{\alpha}_{C}=\liminf_{\varepsilon\rightarrow 0} \frac{\log{\Lambda(Z,\varepsilon)}}{\log{(1/ \eta(\varepsilon))}},~
\overline{\alpha}_{C}=\limsup_{\varepsilon\rightarrow 0} \frac{\log{\Lambda(Z,\varepsilon)}}{\log{(1/ \eta(\varepsilon))}}.
\]

\subsection{Words and sequences}

Let $F_m^+$ be the set of all finite words of symbols $0,1,\ldots,m-1$. For any $w \in F_m^+$, $ |w|$ stands for the length of $w$, that is, the digits of symbols in $w$. Denote $F_m^+(n)=\{w\in F_m^+: |w|=n, n \in \mathbb{N}\}$. Obviously, $F_m^+$ with respect to the law of composition is a free semigroup with $m$ generators. We write ${w}^{\prime}\leq{w}$ if there exists a word $w^{\prime\prime}\in{F_m^+}$ such that $w=w^{\prime\prime}w^{\prime}$. For $w=i_1\ldots i_k \in F_m^{+}$, denote $\overline{w}=i_k\ldots i_1$.

Let $\Sigma_m$ be the set of all two-side infinite sequences of symbols $0,1,\ldots,m-1$, that is,
\[
\Sigma_m=\{\omega=(\ldots, i_{-1}, i_0, i_1, \ldots) : i_j=0,1,\ldots,m-1~{\rm for~all~integer}~j\}.
\]
The metric on $\Sigma_{m}$ is defined by
\[
d'(\omega,\omega^{\prime})=1/{2^{k}},~{\rm where}~k=\inf\{\left|n\right| : i_{n}\neq{i_{n}^{\prime}}\}.
\]
Obviously, $\Sigma_m$ is compact with respect to this metric. The Bernoulli shift $\sigma_{m}:\Sigma_m\rightarrow\Sigma_m$ is a homeomorphism of $\Sigma_m$ given by the formula
\[
(\sigma_{m}\omega)_{k}=i_{k+1}.
\]

Suppose that $\omega\in\Sigma_{m},w\in{F_{m}^{+}},a,b$ are integers, and $a\leq{b}$. We write $\omega|_{[a,b]}=w$ if
$w=i_{a}i_{a+1} \ldots i_{b-1}i_{b}$.

Let $\Sigma_m^+$ be the set of all one-side infinite sequences of symbols $0,1,\ldots,m-1$:
\[
\Sigma_m^+=\{\omega=(i_0, i_1, \ldots) : i_j=0,1,\ldots,m-1~{\rm for~all~integer}~j\}.
\]

\subsection{Hausdorff dimension}
Let $ (X,d) $ be a metric space and $\mathcal{D}(Z,r)$ denote the collection of countable open covers $\{U_{i}\}_{i=1}^{\infty}$ of $Z$ for which $diam U_{i} < r$ for all $i$. Given a subset $ Z \subset X, t\geq 0$, define
\[
m_{H}(Z,t, r)=\inf_{\mathcal{D}(Z,r)}\left\{\sum\limits_{U_{i} \in \mathcal{D}(Z,r)} ({\rm diam}(U_{i}))^t \right\}.
\]
When $r$ decreases, $m_{H}(Z,t, r)$ increases. Therefore there exists the limit
\[
m_{H}(Z,t)=\lim\limits_{r\rightarrow 0} m_{H}(Z,t, r)
\]
which is called the \emph {t-dimensional Hausdorff measure} of $Z$.
The number
\[
dim_{H}(Z)=\inf\{t >0: m_{H}(Z,t)=0\}=\sup\{t>0 : m_{H}(Z,t)=\infty\}
\]
is called the \emph {Hausdorff dimension} of $Z$.

One may equivalently define Hausdorff dimension of using covers by open balls rather than arbitrary open sets. Let $\mathcal{D}^{b}(Z_{k},r)$ denotes the collection of countable open balls covers $\{B(x_{i},r_{i})\}_{i=1}^{\infty}$ of $Z_{k}$ for which $r_{i} < r$ for all $i$, and then define
\[
m^{b}_{H}(Z,t,r)=\inf_{\mathcal{D}^{b}(Z,r)}\left\{\sum\limits_{B(x_{i},r_{i}) \in \mathcal{D}^{b}(Z,r)} ({\rm diam} B(x_{i},r_{i}))^t\right\}.
\]
Finally, define $m^{b}_{H}(Z,t)$ and $dim_{H}^{b}Z$ by the same procedure as above. Then Climenhaga (\cite{Climenhaga1}) showed that $dim_{H}^{b}Z=dim_{H}Z$.

\section{\emph{Topological pressure, lower and upper capacity topological pressure of a free semigroup and their properties}}

In this section, we introduce the definitions of topological pressure, lower and upper capacity topological pressure of a free semigroup action by using C-P structure and provide some properties of them.

\subsection{Topological pressure and lower and upper capacity topological pressure}

Let $X$ be a compact metric space with metric $d$, $f_0,f_1,\ldots,f_{m-1}$ continuous transformations from $X$ to itself. Suppose that a free semigroup with $m$ generators $G=\{f_0,f_1,\ldots,f_{m-1}\}$ acts on X. Given $\varphi_{0},\varphi_{1},\cdots,\varphi_{m-1} \in C(X,\mathbb{R})$, denote $\Phi=\{\varphi_{0},\varphi_{1},\cdots,\varphi_{m-1} \}$. Let $w=i_1i_2 \ldots i_n \in F_m^+$, where $i_j=0,1,\ldots,m-1$ for all $j=1,\ldots,n$ and $f_w=f_{i_1}\circ f_{i_2}\circ \ldots \circ f_{i_n}$. Obviously, $f_{ww'}=f_wf_{w'}$.
 For $w=i_{1}i_{2}\cdots i_{n} \in F_{m}^{+}$, denote
\[
S_{w}\Phi(x):=\varphi_{i_{1}}(x)+\varphi_{i_{2}}(f_{i_{1}}(x))+\cdots+\varphi_{ i_{n}}(f_{i_{n-1}i_{n-2}\cdots i_{1}}(x)).
\]

Considering a finite open cover $\mathcal{U}$ of $X$, let
\[
\mathcal{\mathcal{S}}_{n+1}(\mathcal{U}):=\{\textbf{U}=(U_{0},U_{1},\ldots,U_{n}):\textbf{U}\in \mathcal{U}^{n+1}\},
\]
where $\mathcal{U}^{n+1}=\prod\limits_{i=1}^{n+1}\mathcal{U}$ and $n \geq 0$. For any string $\textbf{U}\in \mathcal{\mathcal{S}}_{n+1}(\mathcal{U})$, define the length of $\textbf{U}$ as $m(\textbf{U}):=n+1$. We put $\mathcal{S}=\mathcal{S}(\mathcal{U})= \cup_{n\geq1} \mathcal{S}_{n}(\mathcal{U})$. For any $\omega=(i_{1},i_{2},\ldots,i_{n},\ldots)\in\Sigma_m^+$,  $n\geq1$, and a given string $\textbf{U}=(U_0,U_1,\ldots,U_{n}) \in \mathcal{S}_{n+1}(\mathcal{U})$, we associate the set
\[
 X_{\omega}(\textbf{U})=\{x\in X : x\in U_{0}, f_{i_{j}}\circ\ldots\circ f_{i_{1}}(x)\in U_j, j=1,2,\ldots,n\}.
 \]
If $w_{\textbf{U}}=\omega|_{[0,n-1]}=i_{1}i_{2}\cdots i_{n} \in F_m^+$, we also denote $ X_{\omega}(\textbf{U})$ by $X_{w_{\textbf{U}}}(\textbf{U})$ for the sake of convenience.
Define
\[
\mathcal{F}=\{X_{\omega}(\textbf{U}): \textbf{U}\in \mathcal{\mathcal{S}}(\mathcal{U})~{\rm and}~\omega \in \Sigma_m^+\},
\]
and three $\xi,\eta,\psi :\mathcal{S}\rightarrow\mathbb{R}^+$  as follows
\[
\xi(\textbf{U}) =\exp\left(\sup_{x \in X_{w_{\textbf{U}}}(\textbf{U})} S_{w_{\textbf{U}}}\Phi(x) \right),
\]
\[ ~\eta(\textbf{U})=\exp(-m(\textbf{U})), ~\psi(\textbf{U})=m(\textbf{U})^{-1}.
\]

It is easy to verify that the set $\mathcal{\mathcal{S}}$, collection of subsets $ \mathcal{F},$ and the functions $\xi, \eta ,$ and $\psi $ satisfy the conditions (1), (2) and (3) in section 3.1 and hence they determine a C-P structure $\tau =(\mathcal{\mathcal{S}},\mathcal{F},\xi,\eta,\psi)$ on $X$.

Given $w \in F_m^{+}, |w|=N, Z\subset X$ and $\alpha \in \mathbb{R}$, we define
\begin{align*}
M_{w}(Z,\alpha,\Phi,\mathcal{U},N):&=\inf\limits_{\mathcal{G}_{w}}\left\{\sum\limits_{\textbf{U}\in\mathcal{G}_{w}} \xi(\textbf{U}) \eta(\textbf{U})^\alpha\right\}\\
&=\inf\limits_{\mathcal{G}_{w}}\left\{\sum\limits_{\textbf{U}\in\mathcal{G}_{w}}\exp\left(-\alpha m(\textbf{U})+\sup_{x \in X_{w_{\textbf{U}}}(\textbf{U})} S_{w_{\textbf{U}}}\Phi(x)\right)\right\},
\end{align*}
where the infimum is taken over all finite or countable collections of strings $\mathcal{G}_{w} \subset \mathcal{\mathcal{S}}(\mathcal{U})$ such that $m(\textbf{U})\ge N+1$ for all $\textbf{U}\in \mathcal{G}_{w}$ and $ \mathcal{G}_{w}$ covers $Z$
(i.e. for any $\textbf{U}\in \mathcal{G}_{w}$, there is $ w_{\textbf{U}} \in F^+_m $ such that $  \overline{w}\leq \overline{w_{\textbf{U}}}$ and $ \bigcup\limits_{\textbf{U}\in \mathcal{G}_{w}}X_{w_{\textbf{U}}}(\textbf{U})\supset Z$).

Let
 \[
 M(Z,\alpha,\Phi,\mathcal{U},N)=\frac{1}{m^{N}}\sum\limits_{|w|=N}M_{w}(Z,\alpha,\Phi,\mathcal{U},N).
 \]

We can easily verify that the function $M(Z,\alpha,\Phi,\mathcal{U},N)$ is non-decreasing as $N$ increases.
Therefore there exists the limit
\[
m(Z,\alpha,\Phi,\mathcal{U})=\lim\limits_{N\to\infty}M(Z,\alpha,\Phi,\mathcal{U},N).
\]

Furthermore, given $ w \in F_m^+ $ and $ |w|=N $, by the Condition $(3^{'})$ in section 3.1, we can define
\begin{align*}
R_{w}(Z,\alpha,\Phi,\mathcal{U},N):&=\inf\limits_{\mathcal{G}_{w}}\left\{\sum\limits_{\textbf{U}\in\mathcal{G}_{w}}\exp\left(-\alpha (N+1)+\sup_{x \in X_{w}(\textbf{U})} S_{w}\Phi(x)\right)\right\}\\
&=\Lambda_{w}(Z,\Phi,\mathcal{U},N)\exp(-\alpha (N+1)),
\end{align*}
where $\Lambda_{w}(Z,\Phi,\mathcal{U},N)=\inf\limits_{\mathcal{G}_{w}}\left\{\sum\limits_{\textbf{U}\in\mathcal{G}_{w}}\exp\left(\sup\limits_{x \in X_{w}(\textbf{U})} S_{w}\Phi(x)\right)\right\},$
the infimum is taken over all finite or countable collections of strings  $\mathcal{G}_{w} \subset \mathcal{\mathcal{S}}(\mathcal{U})$ such that $m(\textbf{U})=N+1$ for all $\textbf{U}\in \mathcal{G}_{w}$ and $ \mathcal{G}_{w}$ covers $Z$ (i.e., for any $\textbf{U}\in \mathcal{G}_{w}$, there is $ w_{\textbf{U}} \in F^+_m $ such that $ w_{\textbf{U}}=w $ and $ \bigcup\limits_{\textbf{U}\in \mathcal{G}_{w}}X_{w_{\textbf{U}}}(\textbf{U})\supset Z$).

Let  \[R(Z,\alpha,\Phi,\mathcal{U},N)=\frac{1}{m^{N}}\sum\limits_{|w|=N}R_{w}(Z,\alpha,\Phi,\mathcal{U},N),\]
and
\[\Lambda(Z,\Phi,\mathcal{U},N)=\frac{1}{m^{N}}\sum\limits_{|w|=N}\Lambda_{w}(Z,\Phi,\mathcal{U},N).\]

It is easy to see that
\[
R(Z,\alpha,\Phi,\mathcal{U},N)=\Lambda(Z,\Phi,\mathcal{U},N)\exp(-\alpha(N+1)).
\]

Set
\[ \underline{r}(Z,\alpha,\Phi,\mathcal{U})=\liminf_{N\to\infty}R(Z,\alpha,\Phi,\mathcal{U},N),\]
\[ \overline{r}(Z,\alpha,\Phi,\mathcal{U})=\limsup_{N\to\infty}R(Z,\alpha,\Phi,\mathcal{U},N).\]

 The C-P structure $\tau$ generates the Carath\'{e}odory-Pesin dimension of $Z$ and the lower and upper Carath\'{e}odory-Pesin capacities of $Z$ with respect to $G$. We denote them by $P_{Z}(G,\Phi,\mathcal{U}), \underline{CP}_{Z}(G,\Phi,\mathcal{U})$ and $\overline{CP}_{Z}(G,\Phi,\mathcal{U})$ respectively. We have
\[
P_{Z}(G,\Phi,\mathcal{U})=\inf\{\alpha:m(Z,\alpha,\Phi,\mathcal{U})=0\}
=\sup\{\alpha:m(Z,\alpha,\Phi,\mathcal{U})=\infty\},
\]
\[
\underline{CP}_{Z}(G,\Phi,\mathcal{U})=\inf\{\alpha: \underline{r}(Z,\alpha,\Phi,\mathcal{U})=0\}=\sup\{\alpha: \underline{r}(Z,\alpha,\Phi,\mathcal{U})=\infty\},
\]
and
\[
\overline{CP}_{Z}(G,\Phi,\mathcal{U})=\inf\{\alpha:\overline{r}(Z,\alpha,\Phi,\mathcal{U})=0\}
=\sup\{\alpha:\overline{r}(Z,\alpha,\Phi,\mathcal{U})=\infty\}.
\]

\begin{theorem}\label{3.1t}
For any set $Z\subset X$, the following limits exist:

\[P_{Z}(G,\Phi)=\lim\limits_{|\mathcal{U}| \rightarrow 0}P_{Z}(G,\Phi,\mathcal{U}),\]

\[\underline{CP}_{Z}(G,\Phi)=\lim\limits_{|\mathcal{U}| \rightarrow 0}\underline{CP}_{Z}(G,\Phi,\mathcal{U}),\]

\[\overline{CP}_{Z}(G,\Phi)=\lim\limits_{|\mathcal{U}| \rightarrow 0}\overline{CP}_{Z}(G,\Phi,\mathcal{U}).\]

\begin{proof}
We use the analogous method as that of \cite{Pesin}. Let $\mathcal{V}$ be a finite open cover of $X $ with diameter smaller than the Lebesgue number of  $\mathcal{U}$.
Then each element $V\in \mathcal{V}$ is contained in some element $U(V ) \in \mathcal{U}$.
For any $\textbf{V}= (V_{0}, V_{1},\ldots, V_{k})\in \mathcal{\mathcal{S}}_{k+1}(\mathcal{V}) \subset \mathcal{\mathcal{S}}(\mathcal{V})$,
we associate the string $\textbf{U}(\textbf{V})=(U(V_{0}),U(V_{1}),\ldots,U(V_{k}))\in \mathcal{\mathcal{S}}_{k+1}(\mathcal{U}) \subset \mathcal{\mathcal{S}}(\mathcal{U})$. Let $w=i_{1}i_{2} \ldots i_{N} \in F_{m}^+.$
If $\mathcal{G}_{w}\subset \mathcal{\mathcal{S}}(\mathcal{V})$
covers $Z$, for each $\textbf{V} \in \mathcal{G}_{w}, m(\textbf{V})\geq N+1$ and there is $ w_{\textbf{V}} \in F^+_m$ such that $ \overline{w} \leq \overline{w_{\textbf{V}}}.$ We denote the word that corresponds to $ \textbf{U}(\textbf{V}) $ by $ w_{\textbf{U}(\textbf{V})}$ such that $w_{\textbf{U}(\textbf{V})}=w_{\textbf{V}} $, then $\textbf{U}(\mathcal{G}_{w}) = \{\textbf{U}(\textbf{V}) : \textbf{V} \in \mathcal{G}_{w}\} \subset \mathcal{\mathcal{S}}(\mathcal{U})$ also covers $Z$. Let
\[
\gamma=\gamma(\mathcal{U}):=\sup\{|\varphi_{j}(x)-\varphi_{j}(y)|: x,y \in U~for~some~U \in \mathcal{U}~and~j=0,1,\cdots,m-1\}.
\]
It follows that
\[
\sup_{x \in X_{w_{\textbf{U}(\textbf{V})}}(\textbf{U}(\textbf{V}))} S_{w_{\textbf{U}(\textbf{V})}}\Phi (x) \leq \sup_{y \in X_{w_{\textbf{V}}}(\textbf{V})} S_{w_{\textbf{V}}}\Phi(y)+\gamma m(\textbf{V}).
\]
Note that $m(\textbf{V})=m(\textbf{U}(\textbf{V}))$. Then for every  $\alpha>0$ and $N>0$. One can easily see that
\[
M_{w}(Z,\alpha,\Phi,\mathcal{U},N) \le  M_{w}(Z,\alpha-\gamma,\Phi,\mathcal{V},N).
\]
Then
\[
M(Z,\alpha,\Phi,\mathcal{U},N) \le M(Z,\alpha-\gamma,\Phi,\mathcal{V},N).
\]
Moreover,
\[
m(Z,\alpha,\Phi,\mathcal{U}) \leq  m(Z,\alpha-\gamma,\Phi,\mathcal{V}).
\]
This implies that
\[
P_{Z}(G,\Phi,\mathcal{U})-\gamma \leq P_{Z}(G,\Phi,\mathcal{V}).
\]
Since $X$ is compact, it has finite open covers of arbitrarily small diameter. Therefore,
\[
P_{Z}(G,\Phi,\mathcal{U})-\gamma \leq \liminf_{|\mathcal{V}| \rightarrow 0}P_{Z}(G,\Phi,\mathcal{V}).
\]
If ${|\mathcal{U}| \rightarrow 0}$, then $\gamma \rightarrow 0$ and hence
\[
 \limsup_{|\mathcal{U}| \rightarrow 0}P_{Z}(G,\Phi,\mathcal{U})  \leq \liminf_{|\mathcal{V}| \rightarrow 0}P_{Z}(G,\Phi,\mathcal{V}).
\]
This implies the existence of the first limit. The existence of the two other limits can be proved in a similar fashion.
\end{proof}
\end{theorem}

The quantities $P_{Z}(G,\Phi)$, $\underline{CP}_{Z}(G,\Phi),$ and $\overline{CP}_{Z}(G,\Phi)$  are called the topological pressure and lower and upper capacity topological pressure of $G$ with respect to $\Phi$ on the set $Z$ respectively.

\begin{remark}
(1) It is easy to see that $P_{Z}(G,\Phi)\le \underline{CP}_{Z}(G,\Phi)\le \overline{CP}_{Z}(G,\Phi)$.

(2) If $\varphi_{i}(x)=0, i=0,1,\cdots,m-1$, we can obtain
$P_{Z}(G, 0)=h_{Z}(G)$, where $h_{Z}(G)$ is the topological entropy on $Z$ in \cite{Ju}. Similarly, the lower and upper capacity topological pressure is the corresponding the lower and upper capacity topological entropies in \cite{Ju}.

(3) When $m=1$, i.e., $G=\{f\}$  and $\Phi=\{\varphi\}$, we obtain $P_{Z}(G,\Phi)=P_{Z}(\varphi)$, $\underline{CP}_{Z}(G,\Phi)=\underline{CP}_{Z}(\varphi)$, $\overline{CP}_{Z}(G,\Phi)=\overline{CP}_{Z}(\varphi)$, for any set $Z \subset X$, where $P_{Z}(\varphi)$, $\underline{CP}_{Z}(\varphi)$ and $\overline{CP}_{Z}(\varphi)$ are denoted by respectively the topological pressure and lower and upper capacity topological pressure in \cite{Pesin}. That is to say, this definition agrees with Pesin's \cite{Pesin}. Moveover, if $Z=X$, then $P_{X}(f,\varphi)=\underline{CP}_{X}(f,\varphi)=\overline{CP}_{X}(f,\varphi)=P(f,\varphi)$, which is equivalent to the classical topological pressure in \cite{Walters}.
\end{remark}

\subsection{Properties of topological pressure and lower and upper capacity topological pressure}

Using the basic properties of the Carath\'{e}odory-Pesin  dimension \cite{Pesin} and definitions, we get the following basic properties of topological pressure and lower and upper capacity topological pressure of a free semigroup action.

\begin{proposition}\label{3.1p}
(1) $P_{\emptyset}(G,\Phi)\le 0.$

(2) $P_{Z_{1}}(G,\Phi)\le P_{Z_{2}}(G,\Phi) ~ if~ Z_{1}\subset Z_{2}$.

(3) $P_{Z}(G,\Phi)=\sup\limits_{i\ge1}P_{Z_{i}}(G,\Phi)~ where~Z=\cup_{i\ge1}Z_{i} ~ and~Z_{i}\subset X , i=1,2,\cdots $.
\end{proposition}

\begin{proposition}\label{3.2p}
(1) $\underline{CP}_{\emptyset}(G,\Phi)\le 0,~\overline{CP}_{\emptyset}(G,\Phi)\le 0.$

(2) $\underline{CP}_{Z_{1}}(G,\Phi)\le \underline{CP}_{Z_{2}}(G,\Phi) ~and~ \overline{CP}_{Z_{1}}(G,\Phi)\le  \overline{CP}_{Z_{2}}(G,\Phi) ~if~ Z_{1}\subset Z_{2}$.

(3) $\underline{CP}_{Z}(G,\Phi) \ge \sup\limits_{i\ge1}\underline{CP}_{Z_{i}}(G,\Phi)$ and $\overline{CP}_{Z}(G,\Phi) \ge \sup\limits_{i\ge1}\overline{CP}_{Z_{i}}(G,\Phi)$,
where $Z=\cup_{i\ge1}Z_{i} ~and~Z_{i}\subset X , i=1,2,\cdots $.

(4) If $g:X \rightarrow X$ is a homeomorphism which commutes with G (i.e., $f_{i}\circ g=g \circ f_{i}$,
for all $f_{i}\in \{f_0,f_1,\ldots,f_{m-1}\}$ ), then
\[
P_{Z}(G,\Phi)=P_{g(Z)}(G,\Phi \circ g^{-1});
\]
\[
\underline{CP}_{Z}(G,\Phi)=\underline{CP}_{g(Z)}(G,\Phi \circ g^{-1});
\]
\[
\overline{CP}_{Z}(G,\Phi)=\overline{CP}_{g(Z)}(G,\Phi \circ g^{-1});
\]
where $\Phi \circ g^{-1}=\{\varphi_{0}\circ g^{-1},\varphi_{1}\circ g^{-1},\cdots,\varphi_{m-1} \circ g^{-1}\}$.
\end{proposition}

Obviously, the function $\eta$ and $\psi$  satisfy Condition (4) in section 3.1. Therefore, similar to  the Theorem 2.2 in \cite{Pesin}, we obtain the following lemma.

\begin{lemma}\label{3.2l}
For any open cover $\mathcal{U}$ of $X$ and any set $Z\subset X$, there exist the limits
\[
\underline{CP}_{Z}(G,\Phi,\mathcal{U})=\liminf_{N\rightarrow \infty}\frac{\log\Lambda(Z,\Phi,\mathcal{U},N)}{N},
\]
\[
\overline{CP}_{Z}(G,\Phi,\mathcal{U})=\limsup_{N\rightarrow \infty}\frac{\log\Lambda(Z,\Phi,\mathcal{U},N)}{N}.
\]
\end{lemma}

\begin{proof}
We will prove the first equality; the second one can be proved in a similar fashion. Put
\[
\alpha=\underline{CP}_{Z}(G,\Phi,\mathcal{U}), ~\beta=\liminf_{N\rightarrow \infty}\frac{\log\Lambda(Z,\Phi,\mathcal{U},N)}{N}.
\]
Given $\gamma>0,$ choose a sequence $N_{i}\rightarrow \infty $ such that
\[0= \underline{r}(Z,\alpha+\gamma,\Phi,\mathcal{U})=\lim\limits_{i \rightarrow \infty}R(Z,\alpha+\gamma,\Phi,\mathcal{U},N_{i}).\]
It follows that $R(Z,\alpha+\gamma,\Phi,\mathcal{U},N_{i})\le1$ for all sufficiently large $i$. Therefore, for such $i$
\begin{equation}
\Lambda(Z,\Phi,\mathcal{U},N_{i})\exp(-(\alpha+\gamma)(N_{i}+1))\le1
\end{equation}
Moreover,
\[
\alpha+\gamma\ge\frac{\log\Lambda(Z,\Phi,\mathcal{U},N_{i})}{N_{i}+1}.
\]
Therefore,
\[
\alpha+\gamma\ge \liminf_{N\rightarrow \infty}\frac{\log\Lambda(Z,\Phi,\mathcal{U},N)}{N}.
\]
Hence,
\begin{equation}
\alpha\ge\beta-\gamma.
\end{equation}
Let us now choose a sequence $N_{i}^{'} $ such that
\[\beta=\lim\limits_{i \rightarrow\infty}\frac{\log\Lambda(Z,\Phi,\mathcal{U},N_{i}^{'})}{N_{i}^{'}}.\]
We have that
\[
\lim\limits_{i \rightarrow\infty}R(Z,\alpha-\gamma,\Phi,\mathcal{U},N_{i}^{'})\ge\underline{r}(Z,\alpha-\gamma,\Phi,\mathcal{U})
=\infty.
\]
This implies that $R(Z,\alpha-\gamma,\Phi,\mathcal{U},N_{i}^{'})\ge1$ for all sufficiently large $i$. Therefore, for such $i$
\[
\Lambda(Z,\Phi,\mathcal{U},N_{i}^{'})\exp(-(\alpha-\gamma)(N_{i}^{'}+1))\ge1
\]
and hence
\[
\alpha-\gamma\le\frac{\log\Lambda(Z,\Phi,\mathcal{U},N_{i}^{'})}{N_{i}^{'}+1}.
\]
Taking the limit as $i \rightarrow\infty$ we obtain that
\[\alpha-\gamma\le\liminf_{N\rightarrow \infty}\frac{\log\Lambda(Z,\Phi,\mathcal{U},N)}{N}=\beta,\]
and consequently,
\begin{equation}
\alpha\le\beta+\gamma.
\end{equation}
Since $\gamma$ can be chosen arbitrarily small, the inequalities (4.2) and (4.3) imply that $\alpha=\beta.$
\end{proof}

\begin{remark}\label{4.2r}
By the Theorem \ref{3.1t} and Lemma \ref{3.2l}, we can obtain
\[
\underline{CP}_{Z}(G,\Phi)=\lim\limits_{|\mathcal{U}| \rightarrow 0}\liminf_{N\rightarrow \infty}\frac{\log\Lambda(Z,\Phi,\mathcal{U},N)}{N},
\]
\[
\overline{CP}_{Z}(G,\Phi)=\lim\limits_{|\mathcal{U}| \rightarrow 0}\limsup_{N\rightarrow \infty}\frac{\log\Lambda(Z,\Phi,\mathcal{U},N)}{N}.
\]
\end{remark}

\begin{theorem}\label{3.3t}
If $\Phi=\{\varphi_{0},\varphi_{1},\cdots, \varphi_{m-1}\}$ and $\Psi=\{\psi_{0},\psi_{1},\cdots, \psi_{m-1}\}$,  we have
\[
|P_{Z}(G,\Phi)-P_{Z}(G,\Psi)| \leq \max_{0 \leq i \leq m-1} \|\varphi_{i}-\psi_{i}\|,
\]
\[
|\underline{CP}_{Z}(G,\Phi)-\underline{CP}_{Z}(G,\Psi)| \leq \max_{0 \leq i \leq m-1} \|\varphi_{i}-\psi_{i}\|,
\]
\[
|\overline{CP}_{Z}(G,\Phi)-\overline{CP}_{Z}(G,\Psi)| \leq \max_{0 \leq i \leq m-1} \|\varphi_{i}-\psi_{i}\|,
\]
where $\|\cdot\|$ denotes the supremum norm in the space of continuous functions on X.
\end{theorem}

\begin{proof}
 Notice that for any $w\in F_m^+$ and $ |w|=N $,
 \[
 \sup_{x \in X}|S_{w}\Phi(x)-S_{w}\Psi(x)| \leq N \max_{0 \leq i \leq m-1} \|\varphi_{i}-\psi_{i}\|.
 \]
 Then we get for any $w\in F_m^+$ and $ |w|\geq N $,
 \[
 \sup_{x \in X}S_{w}\Phi(x)-\sup_{x \in X}S_{w}\Psi(x) \leq \sup_{x \in X}|S_{w}\Phi(x)-S_{w}\Psi(x)| \leq |w| \max_{0 \leq i \leq m-1} \|\varphi_{i}-\psi_{i}\|.
 \]
Hence
\begin{align*}
 M_{w}(Z,\alpha,\Phi,\mathcal{U},N)
 \leq  M_{w}(Z,\alpha-\max_{0 \leq i \leq m-1}\|\varphi_{i}-\psi_{i}\|,\Psi,\mathcal{U},N),
 \end{align*}
 Similarly, we have
 \[
  M_{w}(Z,\alpha,\Phi,\mathcal{U},N) \geq M_{w}(Z,\alpha+\max_{0 \leq i \leq m-1}\|\varphi_{i}-\psi_{i}\|,\Psi,\mathcal{U},N) .
 \]
Therefore,
 \begin{align*}
 M(Z,\alpha+\max_{0 \leq i \leq m-1}\|\varphi_{i}-\psi_{i}\|,\Psi,\mathcal{U},N) &\leq M(Z,\alpha,\Phi,\mathcal{U},N)\\
 &\leq   M(Z,\alpha-\max_{0 \leq i \leq m-1}\|\varphi_{i}-\psi_{i}\|,\Psi,\mathcal{U},N).
 \end{align*}
Taking limit $N \rightarrow \infty$ yields
 \begin{align*}
  m(Z,\alpha+\max_{0 \leq i \leq m-1}\|\varphi_{i}-\psi_{i}\|,\Psi,\mathcal{U})&\leq m(Z,\alpha,\Phi,\mathcal{U})\\
 &\leq  m(Z,\alpha-\max_{0 \leq i \leq m-1}\|\varphi_{i}-\psi_{i}\|,\Psi,\mathcal{U}).
 \end{align*}
Thus
\[
P_{Z}(G,\Psi,\mathcal{U})-\max_{0 \leq i \leq m-1}\|\varphi_{i}-\psi_{i}\| \leq P_{Z}(G,\Phi,\mathcal{U}) \leq P_{Z}(G,\Psi,\mathcal{U})+\max_{0 \leq i \leq m-1}\|\varphi_{i}-\psi_{i}\|.
\]
Let $|\mathcal{U}|\rightarrow 0$, and we obtain
\[
P_{Z}(G,\Psi)-\max_{0 \leq i \leq m-1}\|\varphi_{i}-\psi_{i}\| \leq P_{Z}(G,\Phi) \leq P_{Z}(G,\Psi)+\max_{0 \leq i \leq m-1}\|\varphi_{i}-\psi_{i}\|,
\]
which establishes the first inequality. The proof of the two other inequalities is similar.
\end{proof}

For a free semigroup with $m$ generators acting on $X$, denoting the maps corresponding to the generators by $G=\{ f_0,f_1,\ldots,f_{m-1}\}$,
a set $Z\subset X$ is called $G$-invariant if $f_{i}^{-1}(Z) = Z$  for all $f_{i}\in G$.
For invariant sets, similar to the lower and upper capacity topological pressure of a single map \cite{Pesin}, we have the following theorems.

\begin{theorem}\label{3.4t}
For any $G$-invariant set $Z\subset X$, we have
\[
\underline{CP}_{Z}(G,\Phi)=  \overline{CP}_{Z}(G,\Phi).
\]
Moreover, for any open cover $\mathcal{U}$ of $X$, we have
\[
\underline{CP}_{Z}(\mathcal{U},G,\Phi) = \overline{CP}_{Z}(\mathcal{U},G,\Phi).
\]
\end{theorem}

\begin{proof}
Let $Z\subset X$ be a $G$-invariant set.
For any $w^{(1)}, w^{(2)}\in F_m^+$ where $ |w^{(1)}|=p $ and $|w^{(2)}|=q$, we choose two collections of strings $\mathcal{G}_{w^{(1)}}\subset \mathcal{\mathcal{S}}_{p+1}(\mathcal{U})$ and $\mathcal{G}_{w^{(2)}}\subset \mathcal{\mathcal{S}}_{q+1}(\mathcal{U})$ which cover $Z$.
Supposing that $\textbf{U}=(U_{0},U_{1}\ldots,U_{p})\in \mathcal{G}_{w^{(1)}}$ and $\textbf{V}=(V_{0},V_{1}\ldots,V_{q})\in \mathcal{G}_{w^{(2)}}$,
we define
\begin{center}
$\textbf{U}\textbf{V}=(U_{0},U_{1},\ldots,U_{p},V_{0},V_{1},\ldots,V_{q}) $.
\end{center}
For a fixed $i\in \{0,1,\ldots,m-1\}$, we consider
\[
\mathcal{G}_{w}:=\{\textbf{U}\textbf{V}:  \textbf{U}\in \mathcal{G}_{w^{(1)}},  \textbf{V}\in \mathcal{G}_{w^{(2)}}\}\subset \mathcal{S}_{p+q+2}(\mathcal{U}),
\]
where $w=w^{(1)}iw^{(2)}$. Then
\[
X_{w}(\textbf{U}\textbf{V})=X_{w^{(1)}}(\textbf{U}) \cap (f_i \circ f_{\overline{w^{(1)}}})^{-1}(X_{w^{(2)}}(\textbf{V})),
\]
and $ m(\textbf{U}\textbf{V})=m(\textbf{U})+m(\textbf{V}).$
Since $Z$ is a $G$-invariant set, the collection of strings $\mathcal{G}_{\omega}$ also covers $Z$. By the definition of $\Lambda_{w}(Z,\Phi,\mathcal{U},p+q+1)$, we have
\begin{align*}
&\Lambda_{w}(Z,\Phi,\mathcal{U},p+q+1) \le \sum\limits_{\textbf{U}\textbf{V}\in\mathcal{G}_{w}}\exp \left(\sup_{x \in X_{w}(\textbf{U}\textbf{V})} S_{w}\Phi (x)\right)\\
&\le  c\times \left\{ \sum\limits_{\textbf{U}\in\mathcal{G}_{w^{(1)}}}\exp \left(\sup_{x \in X_{w^{(1)}}(\textbf{U})} S_{w^{(1)}}\Phi (x)\right)\right\}  \times  \left\{ \sum\limits_{\textbf{V}\in\mathcal{G}_{w^{(2)}}}\exp \left(\sup_{x \in X_{w^{(2)}}(\textbf{V})} S_{w^{(2)}}\Phi(x)\right)\right\},
\end{align*}
where $c=\max_{0\leq i \leq m-1}e^{\|\varphi_{i}\|}$, which implies
\[
\frac{1}{m}\sum\limits_{i=0}^{m-1}\Lambda_{w}(Z,\Phi,\mathcal{U},p+q+1)\le  c\times \Lambda_{w^{(1)}}(Z,\Phi,\mathcal{U},p)\times \Lambda_{w^{(2)}}(Z,\varphi,\mathcal{U},q).
\]
 It follows that
\[\Lambda(Z,\Phi,\mathcal{U},p+q+1)\le c\times \Lambda(Z,\Phi,\mathcal{U},p)\times \Lambda(Z,\Phi,\mathcal{U},q).\]
Let $a_{p}=\log\Lambda(Z,\Phi,\mathcal{U},p)$. Note that $\Lambda(Z,\Phi,\mathcal{U},p) \geq e^{-p\cdot \max_{0\leq i \leq m-1}\|\varphi_{i}\|}$. Therefore, $\inf\limits_{p>1}\frac{a_{p-1}}{p}>-\infty$. The desired result is now a direct consequence of the following Lemma \ref{3.5l} .
\end{proof}

\begin{lemma}\label{3.5l}
Let $\{a_{p}\},~p=1,2,\ldots$ be a sequence of numbers satisfying $\inf\limits_{p>1}\frac{a_{p-1}}{p}>-\infty$ and $a_{p+q+1}\le c'+a_{p}+a_{q}$ for all $p,q > 1$ where $c'>0$ is a constant. Then the $\lim\limits_{p\rightarrow\infty}\frac{a_{p}}{p}$ exists and coincides with $\inf\limits_{p >1}(\frac{a_{p-1}}{p}+\frac{c'}{p})$.
\end{lemma}
\begin{proof}
The proof is similar to the Theorem 4.9 in \cite{Walters}, so we omit the proof.
\end{proof}

\begin{remark}
Indeed, when $Z=X$ and $\Phi=\{\varphi\}$, i.e.,~$\varphi_{0}=\varphi_{1}=\cdots=\varphi_{m-1}=\varphi$, it is easy to get $\underline{CP}_{X}(G,\Phi)=\overline{CP}_{X}(G,\Phi)=P_{X}(G,\varphi)$, where $P_{X}(G,\varphi)$ is denoted by the topological pressure of free semigroup actions on $X$ in \cite{Lin}.
\end{remark}

Next, we discuss the relationship between the topological pressure and upper capacity topological pressure of a free semigroup action generated by $G$ on $Z$ when $Z$ is a compact $G$-invariant set.
Given a compact $G$-invariant set $Z\subset X$ and an open cover $\mathcal{U}$ of $X$, we choose any $\alpha>P_{Z}(G,\Phi,\mathcal{U})$, then
\[
m(Z,\alpha,\Phi,\mathcal{U})=\lim\limits_{N\to\infty}M(Z,\alpha,\Phi,\mathcal{U},N)=0.
\]
Since $M(Z,\alpha,\Phi,\mathcal{U},N)$ is non-decreasing as $N$ increases and non-negative, it follows that $M(Z,\alpha,\Phi,\mathcal{U},N)=0 $ for any $N$. Therefore, for any $ w \in F_{m}^+ $ and $|w|=N$, we have $M_{w}(Z,\alpha,\Phi,\mathcal{U},N)=0 $. For
$ M_{w}(Z,\alpha,\Phi,\mathcal{U},2)=0$, there exists $ \mathcal{A}_{w} \subset \mathcal{S}(\mathcal{U})$ such that $\mathcal{A}_{w} $ covers $Z$ (i.e., for any $\textbf{U} \in \mathcal{A}_{w}$, there exists $w_{\textbf{U}} \in F_{m}^+$ such that $|w_{\textbf{U}}|=m(\textbf{U})-1, \overline{w} \leq \overline{w_{\textbf{U}}}$ and $\bigcup\limits_{\textbf{U}\in \mathcal{A}_{w}}X_{w_{\textbf{U}}}(\textbf{U})\supset Z$) and

\begin{equation}\label{3.6}
Q(Z,\alpha,\Phi,\mathcal{A}_{w}):=\sum\limits_{\textbf{U} \in \mathcal{A}_{w}} \exp\left(-\alpha\cdot m(\textbf{U})+\sup_{x \in X_{w_{\textbf{U}}}(\textbf{U})} S_{w_{\textbf{U}}}\Phi(x)\right) ~< ~p ~< ~1,
\end{equation}
where $p$ is a constant, $c=\max_{0 \leq i\leq m-1} e^{\|\varphi_{i}\|}$ and $c p < 1$.
Since $Z$ is compact we can choose $\mathcal{A}_{w}$ to be finite and $K \geq 3$ to be a constant and
\begin{equation}\label{3.7}
\mathcal{A}_{w} \subset \bigcup\limits_{m=3}^{K} \mathcal{S}_{m}(\mathcal{U}).
\end{equation}
For any $w^{(1)}, w^{(2)} \in F_m^+, |w^{(1)}|=|w^{(2)}|=2$ and $j \in \{0,1,\ldots, m-1\}$,
we can construct
\[
\mathcal{A}_{w^{(1)}jw^{(2)}}=\{\textbf{U}\textbf{V} : \textbf{U} \in \mathcal{A}_{w^{(1)}}~{\rm and}~\textbf{V} \in \mathcal{A}_{w^{(2)}}\},
\]
where $ \mathcal{A}_{w^{(1)}}, \mathcal{A}_{w^{(2)}} $ satisfy (\ref{3.6}), (\ref{3.7}). Then
\[
X(\textbf{U}\textbf{V})=X_{w_{\textbf{U}}}(\textbf{U})\cap(f_j \circ f_{\overline{w_{\textbf{U}}}})^{-1}(X_{w_{\textbf{V}}}(\textbf{V})),
\]
where the word corresponds to $ \textbf{U}\textbf{V} $ is $w_{\textbf{U}}jw_{\textbf{V}}$ and $ m(\textbf{U}\textbf{V})=m(\textbf{U})+m(\textbf{V}) \geq 6$.
Since $Z$ is $G$-invariant, then $\mathcal{A}_{w^{(1)}jw^{(2)}}$ covers $ Z $. It is easy to see that
\[
Q(Z,\alpha,\Phi,\mathcal{A}_{w^{(1)}jw^{(2)}})\leq c \cdot Q(Z,\alpha,\Phi,\mathcal{A}_{w^{(1)}})\cdot  Q(Z,\alpha,\Phi,\mathcal{A}_{w^{(2)}})<c p^2.
\]
By mathematical induction, for each $n \in \mathbb{N}$ and $j_1,\ldots,j_{n-1} \in \{0,1,\ldots, m-1\}$, we can define
$\mathcal{A}_{w^{(1)}j_1w^{(2)}j_2 \ldots w^{(n-1)}j_{n-1}w^{(n)}}$  which covers $Z$ and satisfies
\[
Q(Z,\alpha,\Phi,\mathcal{A}_{w^{(1)}j_1w^{(2)}j_2 \ldots w^{(n-1)}j_{n-1}w^{(n)}})<c^{n-1} p^n.
\]
Let $\Gamma_{w^{(1)}j_1w^{(2)} \ldots}=\mathcal{A}_{w^{(1)}}\cup\mathcal{A}_{w^{(1)}j_1w^{(2)}}\cup \cdots$. Since $Z$ is $G$-invariant, then
$\Gamma_{w^{(1)}j_1w^{(2)} \ldots}$ covers $Z$ and
\[
Q(Z,\alpha,\Phi,\Gamma_{w^{(1)}j_1w^{(2)} \ldots})\leq \sum\limits_{n=1}^\infty (c^{n-1} p^n) < \infty.
\]
Therefore, for any $\omega \in \Sigma_{m}^+$, there exists $\Gamma_{\omega}$ covering $Z$ and $Q(Z,\alpha,\Phi,\Gamma_{\omega}) < \infty.$
Put
\[
\mathcal{F}=\{\Gamma_{\omega} : \omega \in \Sigma_{m}^+\}.
\]

Similar to \cite{Ju}, the following condition is given and the following Theorem \ref{3.7t} holds under this condition.
\begin{condition}\label{3.6c}
For any $N>0$ and any $w=i_1i_2 \ldots i_N \in F_m^+$, there exists $\Gamma_{\omega} \in \mathcal{F}$ such that for any $ \textbf{U} \in \Gamma_{\omega} $, $\overline{w} \leq \overline{w_{\textbf{U}}}$ and $ N+1 \leq m(\textbf{U}) \leq N+K$, where $w_{\textbf{U}}$ is the word corresponds to $\textbf{U}$ and $K$  is given by (\ref{3.7}).
\end{condition}

\begin{theorem}\label{3.7t}
Under the condition \ref{3.6c}, for any compact $G$-invariant set $Z\subset X$, we have
\[
P_{Z}(G,\Phi)=\underline{CP}_{Z}(G,\Phi)= \overline{CP}_{Z}(G,\Phi).
\]
Moreover, for any open cover $\mathcal{U}$ of $X$, we have
\[
P_{Z}(G,\Phi,\mathcal{U}) =\underline{CP}_{Z}(G,\Phi,\mathcal{U},)=\overline{CP}_{Z}(G,\Phi,\mathcal{U}).
\]
\end{theorem}

\begin{proof}
Under the condition \ref{3.6c}, for any $ N >0$ and any $ w=i_1i_2 \ldots i_N \in F_m^+ $, there is $\Gamma_{\omega} \in \mathcal{F}$ covering $Z$ such that for any $\textbf{U} \in \Gamma_{\omega}$, the word corresponds to $\textbf{U}  $ is $ w_{\textbf{U}} $ and $ \overline{w}\leq \overline{w_{\textbf{U}}}.$ Then for any $ x \in Z $, there exists a string $ \textbf{U}=(U_0,U_1,\ldots, U_N,\ldots,U_{N+P}) \in \Gamma_{\omega} $ such that $x \in X_{w_{\textbf{U}}} (\textbf{U})$, where $0\leq P<K$. Let $\textbf{U}^{*}=(U_0,U_1,\ldots, U_N)$. Then $X_{w_{\textbf{U}}} (\textbf{U}) \subset X_{w} (\textbf{U}^{*}) $.
Using $\Gamma_{w}^{*}$ denotes the collection of all substrings $\textbf{U}^{*}$ constructed above
and let
\[
\varepsilon(|\mathcal{U}|):=\sup\{|\varphi_{i}(x)-\varphi_{i}(y)|: x,y \in U ,\forall \ U \in \mathcal{U}, 0\leq i \leq m-1\}.
\]
Because of $\varphi_{i}\in C(X,\mathbb{R})$, $\varphi_{i}$ is uniformly continuous for $0\leq i \leq m-1$. Then $\varepsilon(|\mathcal{U}|)$ is finite and $\lim\limits_{|\mathcal{U}|\rightarrow 0}\varepsilon(|\mathcal{U}|)=0$.
We have
\begin{align*}
\sup_{y \in X_{w}(\textbf{U}^{*})}S_{w}\Phi(y)
 \leq \sup_{x \in X_{w_{\textbf{U}}}(\textbf{U})}S_{w_{\textbf{U}}}\Phi(x)+K\cdot\max_{0\leq i \leq m-1}\|\varphi_{i}\|+N\varepsilon(|\mathcal{U}|).
\end{align*}
Therefore,
\begin{align*}
e^{-(\alpha+\varepsilon(|\mathcal{U}|)) N}\sum_{\textbf{U}^{*} \in \Gamma_{w}^{*}} &\exp \big(\sup_{y \in X_{w}(\textbf{U}^{*})} S_{w}\Phi(y)\big)\\
&\leq \max\{1, e^{\alpha K}\}\cdot e^{K \cdot \max_{0\leq i \leq m-1}\|\varphi_{i}\|}\cdot Q(Z,\alpha,\Phi,\Gamma_{\omega})\\
&<\infty.\\
\end{align*}
It follows that
\[
e^{-(\alpha+\varepsilon(|\mathcal{U}|)) N}\Lambda_{w}(Z,\Phi,\mathcal{U},N)<\infty.
\]
Taking average yields
\[
e^{-(\alpha+\varepsilon(|\mathcal{U}|)) N}\Lambda(Z,\Phi,\mathcal{U},N)<\infty.
\]
By Lemma \ref{3.2l}
\[
\alpha+\varepsilon(|\mathcal{U}|)> \overline{CP}_{Z}(G,\Phi,\mathcal{U}),
\]
that is,
\[
\alpha> \overline{CP}_{Z}(G,\Phi,\mathcal{U})-\varepsilon(|\mathcal{U}|).
\]
Letting $|\mathcal{U}|\rightarrow 0$, we get $P_{Z}(G,\Phi)> \overline{CP}_{Z}(G,\Phi)$.
\end{proof}

\section{\emph{Two equivalent definitions of topological pressure in the present paper}}

Let $(X,d)$ be a compact metric space. Now, we describe two other approaches to redefine the topological pressure and lower and upper capacity topological pressure of $G=\{f_0,\ldots,f_{m-1}\}$ on any subset of $X$ and $\Phi=\{\varphi_{0}, \varphi_{1},\cdots,\varphi_{m-1}\},$ where $f_{i}(i=0,1,\cdots,m-1)$ is continuous and $\varphi_{0}, \varphi_{1},\cdots,\varphi_{m-1} \in C(X,\mathbb{R})$.

For each $w \in F_m^{+}$, a new metric $d_w$ on $X$ (named Bowen metric) is given by
\[d_w(x_1,x_2)=\max_{w' \leq \overline{w}} d(f_{w'}(x_1),f_{w'}(x_2)).\]
Clearly, if $\overline{w'} \leq \overline{w''}$, then $d_{w'}(x_1,x_2) \leq d_{w''}(x_1,x_2)$ for all $x_1,x_2 \in X$.

\subsection{Definition using Bowen balls}
Fix a number $\delta > 0$. Given $w\in F_{m}^+$ and a point $x\in X$, define the $(w,\delta)$-Bowen ball at $x$ by
\[
B_{w}(x,\delta)=\{y \in X : d(f_{w'}(x),f_{w'}(y)) \le \delta, ~{\rm for} ~w' \le \overline{w} \}.
\]
Put $\mathcal{S}=X\times F_{m}^{+} $. We define the collection of subsets
\[
\mathcal{F}=\{B_{w}(x,\delta) : x\in X, w \in F_{m}^+\},
\]
and three functions $\xi,\eta,\psi: \mathcal{S}\rightarrow \mathbb{R} $ as follows
\[
\xi(x,w)=\exp\left(\sup_{y \in B_{w}(x,\delta)}S_{w}\Phi(y)\right),
\]
\[
 ~\eta(x,w)=\exp\left(-(|w|+1)\right),~ \psi(x,w)=(|w|+1)^{-1}.
\]

One can easy verify that the collection of subsets $\mathcal{F}$ and the functions $\xi,\eta,\psi$  satisfy conditions (1), (2), (3) and $(3')$  in section 3.1. Therefore, they determine a C-P structure $\tau =(\mathcal{S},\mathcal{F},\xi,\eta,\psi)$ on $X$.

Given $w \in F_m^{+}, |w|=N, Z\subset X$ and $\alpha \in \mathbb{R}$, we define
\begin{align*}
\overline{M}_{w}(Z,\alpha,\Phi,\delta,N):&=\inf\limits_{\mathcal{G}_{w} }\left\{\sum\limits_{B_{w'}(x,\delta)\in\mathcal{G}_{w}} \xi(x,|w|)\eta(x,|w|)^{\alpha}\right\}\\
&=\inf\limits_{\mathcal{G}_{w} }\left\{\sum\limits_{B_{w'}(x,\delta)\in\mathcal{G}_{w} }\exp\left(-\alpha\cdot (|w'|+1)+\sup_{y \in B_{w'}(x,\delta)}S_ {w'} \Phi(y)\right)\right\},
\end{align*}
where the infimum is taken over all finite or countable subcollections $\mathcal{G}_{w} \subset \mathcal{F}$ covering $Z$ (i.e. for any $ B_{w'}(x,\delta) \in \mathcal{G}_{w}, \overline{w}\leq \overline{w'} $ and $ \bigcup\limits_{B_{w'}(x,\delta) \in \mathcal{G}_{w}} B_{w'}(x,\delta) \supset Z).$

Let
\[
\overline{M}(Z,\alpha,\Phi,\delta,N)=\frac{1}{m^{N}}\sum\limits_{|w|=N}\overline{M}_{w}(Z,\alpha,\Phi,\delta,N).
\]

We can easily verify that the function $\overline{M}(Z,\alpha,\Phi,\delta,N)$ is non-decreasing as $N$ increases.
Therefore, there exists the limit
\[
\overline{m}(Z,\alpha,\Phi,\delta)=\lim\limits_{N\to\infty}\overline{M}(Z,\alpha,\Phi,\delta,N).
\]

Furthermore, by the condition $(3^{'})$ in section 3.1, we can define
\begin{align*}
\overline{R}_{w}(Z,\alpha,\Phi,\delta,N)
&=\inf\limits_{\mathcal{G}_{w}}\left\{\sum\limits_{B_{w}(x,\delta) \in \mathcal{G}_{w}}\exp\left(-\alpha\cdot (N+1)+\sup_{y \in B_{w}(x,\delta)}S_{w} \Phi(y)\right)\right\}\\
&=\overline{\Lambda}_{w}(Z,\Phi,\delta,N)\exp(-\alpha\cdot (N+1)),
\end{align*}
where $\overline{\Lambda}_{w}(Z,\Phi,\delta,N)=\inf\limits_{\mathcal{G}_{w}}\left\{\sum\limits_{B_{w}(x,\delta) \in \mathcal{G}_{w}}\exp\left(\sup\limits_{y \in B_{w}(x,\delta)} S_{w} \Phi(y) \right)\right\},$ the infimum is taken over all finite or countable subcollections $\mathcal{G}_{w} \subset \mathcal{F}$ covering $Z$ and the words corresponds to every ball in $\mathcal{G}_{w}$ are all equal.

Let
\[
\overline{R}(Z,\alpha,\Phi,\delta,N)=\frac{1}{m^{N}}\sum\limits_{|w|=N}\overline{R}_{w}(Z,\alpha,\Phi,\delta,N),
\]
\[
\overline{\Lambda}(Z,\Phi,\delta,N)=\frac{1}{m^{N}}\sum\limits_{|w|=N}\overline{\Lambda}_{w}(Z,\Phi,\delta,N).
\]

We set
\[
\underline{r}(Z,\alpha,\Phi,\delta)=\liminf_{N\to\infty}\overline{R}(Z,\alpha,\Phi,\delta,N),
\]
\[
\overline{r}(Z,\alpha,\Phi,\delta)=\limsup_{N\to\infty}\overline{R}(Z,\alpha,\Phi,\delta,N).
\]

The C-P structure $\tau$ generates the Carath\'{e}odory-Pesin dimension of $Z$ and the lower and upper Carath\'{e}odory-Pesin capacities of $Z$  with respect to $G$. We denote them by $P_{Z}(G,\Phi,\delta), \underline{CP}_{Z}(G,\Phi,\delta),$ and $\overline{CP}_{Z}(G,\Phi,\delta)$ respectively. We have that
\[
P_{Z}(G,\Phi,\delta)=\inf\{\alpha:\overline{m}(Z,\alpha,\Phi,\delta)=0\}=\sup\{\alpha: \overline{m}(Z,\alpha,\Phi,\delta)=\infty\}.
\]
\[
\underline{CP}_{Z}(G,\Phi,\delta)=\inf\{\alpha: \underline{r}(Z,\alpha,\Phi,\delta)=0\}=\sup\{\alpha: \underline{r}(Z,\alpha,\Phi,\delta)=\infty\}.
\]
\[
\overline{CP}_{Z}(G,\Phi,\delta)=\inf\{\alpha: \overline{r}(Z,\alpha,\Phi,\delta)=0\}=\sup\{\alpha: \overline{r}(Z,\alpha,\Phi,\delta)=\infty\}.
\]

\begin{theorem}\label{4.1t}
For any set $Z \subset X$, the following limits exist:
 \[
 P_{Z}(G,\Phi)=\lim\limits_{\delta\rightarrow 0}P_{Z}(G,\Phi,\delta),
 \]
 \[
 \underline{CP}_{Z}(G,\Phi)=\lim\limits_{\delta\rightarrow 0}\underline{CP}_{Z}(G,\Phi,\delta),
 \]
 \[
 \overline{CP}_{Z}(G,\Phi)=\lim\limits_{\delta\rightarrow 0}\overline{CP}_{Z}(G,\Phi,\delta).
 \]
\end{theorem}

\begin{proof}
Let $\mathcal{U}$ be a finite open cover of $X $, and $\delta(\mathcal{U})$ is the Lebesgue number of $\mathcal{U}$. It is easily to see that for any $x\in X$, if $x\in X_{\omega}(\textbf{U})$ for some $\textbf{U}\in \mathcal{S}_{k+1}(\mathcal{U})$ and some $\omega \in \Sigma^+_m $ then
\[
B_{\omega|_{[0,k-1]}}(x,\frac{1}{2}\delta(\mathcal{U}))\subset X_{\omega}(\textbf{U})\subset B_{\omega|_{[0,k-1]}}(x,2|\mathcal{U}|).
\]
It follows from Theorem 4.1 that
\[
 P_{Z}(G,\Phi)=\lim\limits_{\delta\rightarrow 0}P_{Z}(G,\Phi,\delta),
 \]
\[
 \underline{CP}_{Z}(G,\Phi)=\lim\limits_{\delta\rightarrow 0}\underline{CP}_{Z}(G,\Phi,\delta),
 \]
\[
\overline{CP}_{Z}(G,\Phi)=\lim\limits_{\delta\rightarrow 0}\overline{CP}_{Z}(G,\Phi,\delta).
 \]
\end{proof}

\begin{remark}
Similar to Remark \ref{4.2r}, we have
\[
\underline{CP}_{Z}(G,\Phi)=\lim\limits_{\delta \rightarrow 0}\liminf_{N\rightarrow \infty}\frac{\log \overline{\Lambda}(Z,\Phi,\delta,N)}{N},
\]
\[
\overline{CP}_{Z}(G,\Phi)=\lim\limits_{\delta \rightarrow 0}\limsup_{N\rightarrow \infty}\frac{\log \overline{\Lambda}(Z,\Phi,\delta,N)}{N}.
\]
\end{remark}

\subsection{Definition using the center of Bowen's ball }

Put $\mathcal{S}=X\times F_{m}^+$, define the collection of subsets
\[
\mathcal{F}=\{B_{w}(x,\delta) : x\in X, w \in F_{m}^+\}.
\]
We redefine three functions $\xi,\eta,\psi: \mathcal{S}\rightarrow \mathbb{R} $ as follows
\[
\xi(x,w)=\exp\left(S_{w} \Phi(x)\right),
\]
\[
 ~\eta(x,w)=\exp\left(-|w|\right),~ \psi(x,w)=|w|^{-1}.
\]
One can easily verify that the collection of subsets $\mathcal{F}$ and the functions $\xi,\eta,\psi$  satisfy conditions (1), (2), (3) and $(3')$ in section 3.1. Therefore, they determine a C-P structure $\tau =(\mathcal{S},\mathcal{F},\xi,\eta,\psi)$ on $X$.
Given $w \in F_m^{+}, |w|=N, Z\subset X$ and $\alpha \in \mathbb{R}$, we define
\begin{align*}
M'_{w}(Z,\alpha,\Phi,\delta,N):&=\inf\limits_{\mathcal{G}_{w} }\left\{\sum\limits_{B_{w'}(x,\delta)\in\mathcal{G}_{w}} \xi(x,|w|)\eta(x,|w|)^{\alpha}\right\}\\
&=\inf\limits_{\mathcal{G}_{w} }\left\{\sum\limits_{B_{w'}(x,\delta)\in\mathcal{G}_{w} }\exp\bigg(-\alpha\cdot |w'|+S_{w'} \Phi(x)\bigg)\right\},
\end{align*}
where the infimum is taken over all finite or countable subcollections $\mathcal{G}_{w} \subset \mathcal{F}$ covering $ Z $ (i.e., for any $ B_{w'}(x,\delta) \in \mathcal{G}_{w}, \overline{w}\leq \overline{w'} $ and $ \bigcup\limits_{B_{w'}(x,\delta) \in \mathcal{G}_{w}} B_{w'}(x,\delta) \supset Z). $

Let
\[
M'(Z,\alpha,\Phi,\delta,N)=\frac{1}{m^{N}}\sum\limits_{|w|=N} M'_{w}(Z,\alpha,\Phi,\delta,N).
\]

We can easily verify that the function $M'(Z,\alpha,\Phi,\delta,N)$ is non-decreasing as $N$ increases.
Therefore, there exists the limit
\[
m'(Z,\alpha,\Phi,\delta)=\lim\limits_{N\to\infty}M'(Z,\alpha,\Phi,\delta,N).
\]

Furthermore, by the condition $(3^{'})$ in section 3.1, we can define
\begin{align*}
R'_{w}(Z,\alpha,\Phi,\delta,N)
&=\inf\limits_{\mathcal{G}_{w}}\left\{\sum\limits_{B_{w}(x,\delta) \in \mathcal{G}_{w}}\exp\bigg(-\alpha\cdot N+S_{w} \Phi(x)\bigg)\right\}\\
&=\Lambda'_{w}(Z,\Phi,\delta,N)\exp(-\alpha\cdot N),
\end{align*}
where $\Lambda'_{w}(Z,\Phi,\delta,N)=\inf\limits_{\mathcal{G}_{w}}\left\{\sum\limits_{B_{w}(x,\delta) \in \mathcal{G}_{w}}\exp\left( S_{w} \Phi(x)\right)\right\},$
the infimum is taken over all finite or countable subcollections $\mathcal{G}_{w} \subset \mathcal{F}$ covering $Z$ and the words correspond to every ball in $\mathcal{G}_{w}$ is all equal.

Let
\[R'(Z,\alpha,\Phi,\delta,N)=\frac{1}{m^{N}}\sum\limits_{|w|=N}R'_{w}(Z,\alpha,\Phi,\delta,N),\]
\[\Lambda'(Z,\Phi,\delta,N)=\frac{1}{m^{N}}\sum\limits_{|w|=N}\Lambda'_{w}(Z,\Phi,\delta,N).\]

We set
\[ \underline{r'}(Z,\alpha,\Phi,\delta)=\liminf_{N\to\infty}R'(Z,\alpha,\Phi,\delta,N),\]
\[ \overline{r'}(Z,\alpha,\Phi,\delta)=\limsup_{N\to\infty}R'(Z,\alpha,\Phi,\delta,N).\]

The C-P structure $\tau$ generates the Carath\'{e}odory-Pesin dimension of $Z$ and the lower and upper Carath\'{e}odory-Pesin capacities of $Z$  with respect to $G$. We denote them by $P'_{Z}(G,\Phi,\delta), \underline{CP'}_{Z}(G,\Phi,\delta),$ and $\overline{CP'}_{Z}(G,\Phi,\delta)$ respectively. We have that
\[
P'_{Z}(G,\Phi,\delta)=\inf\{\alpha: m'(Z,\alpha,\Phi,\delta)=0\}=\sup\{\alpha: m'(Z,\alpha,\Phi,\delta)=\infty\},
\]
\[
\underline{CP'}_{Z}(G,\Phi,\delta)=\inf\{\alpha: \underline{r'}(Z,\alpha,\Phi,\delta)=0\}=\sup\{\alpha: \underline{r'}(Z,\alpha,\Phi,\delta)=\infty\},
\]
\[
\overline{CP'}_{Z}(G,\Phi,\delta)=\inf\{\alpha: \overline{r'}(Z,\alpha,\Phi,\delta)=0\}=\sup\{\alpha: \underline{r'}(Z,\alpha,\Phi,\delta)=\infty\}.
\]

\begin{theorem}\label{4.2t}
For any set $Z \subset X$, the following limits exist:
 \[P_{Z}(G,\Phi)=\lim\limits_{\delta\rightarrow 0}P'_{Z}(G,\Phi,\delta),\]
 \[\underline{CP}_{Z}(G,\Phi)=\lim\limits_{\delta\rightarrow 0}\underline{CP'}_{Z}(G,\Phi,\delta),\]
 \[\overline{CP}_{Z}(G,\Phi)=\lim\limits_{\delta\rightarrow 0}\overline{CP'}_{Z}(G,\Phi,\delta).\]
\end{theorem}

\begin{proof}
We use the analogous method as that of \cite{Climenhaga1}. On the one hand, given $\delta > 0$, let
\[
\varepsilon(\delta):=\sup \{|\varphi_{i}(x)-\varphi_{i}(y)|: d(x,y)< \delta, i=0,1,\cdots,m-1\}.
\]
and observe that since $\varphi_{i} \in C(X,\mathbb{R})$ and X is compact, $\varphi_{i}$ is in fact uniformly continuous, hence $\varepsilon(\delta)$ is finite and
$\lim\limits_{\delta \rightarrow 0}\varepsilon(\delta)=0$. Furthermore, given $x \in X$, $w \in F^+_m$ and $|w|=N$, then for any $y \in B_{w}(x, \delta)$, we have
\[
\left|\varphi_{i}(f_{w'}(x))-\varphi_{i}(f_{w'}(y))\right| \leq \varepsilon(\delta) ~for ~any~w'\leq \overline{w},~i=0,1,\cdots,m-1.
\]
Thus
\[
\left|S_{w}\Phi(x)-S_{w}\Phi(y)\right| \leq |w|\varepsilon(\delta).
\]
Now fix $\delta > 0$, choose a finite open cover $\mathcal{U}$ of $X $ with $|\mathcal{U}| < \delta$ and
let $\gamma(\mathcal{U})$ be the Lebesgue number of $\mathcal{U}$.  Let $\mathcal{G}_{w}=\{B_{w'}(x, \frac{1}{2}\gamma(\mathcal{U})): x\in X, \overline{w} \leq \overline{w'}\} $ be a open cover of $Z$, then for each
$B_{w'}(x, \frac{1}{2}\gamma(\mathcal{U})) \in \mathcal{G}_{w}$ there exists $\textbf{U} \in \mathcal{\mathcal{S}}_{|w'|+1}(\mathcal{U})$
such that $B_{w'}(x, \frac{1}{2}\gamma(\mathcal{U})) \subset X_{w'}(\textbf{U})$. Set
\[
\mathcal{G'}_{w}=\{\textbf{U}: B_{w'}(x, \frac{1}{2}\gamma(\mathcal{U})) \subset X_{w'}(\textbf{U})\}
\]
and then
\begin{align*}
M_{w}(Z,\alpha,\Phi,\mathcal{U},N)
& \leq \sum\limits_{\textbf{U}\in\mathcal{G'}_{w}}\exp\left(-\alpha m(\textbf{U})+\sup_{y \in X_{w'}(\textbf{U})} S_{w'}\Phi (y)\right)\\
&\leq \sum\limits_{B_{w'}(x,\frac{1}{2}\gamma(\mathcal{U})) \in\mathcal{G}_{w}}\exp\Big( -\alpha m(\textbf{U})+S_{w'}\Phi (x)+|w'|\varepsilon(\delta) \Big)\\
& \leq e^{-\alpha} \cdot \sum\limits_{B_{w'}(x,\frac{1}{2}\gamma(\mathcal{U}))\in\mathcal{G}_{w}}\exp\Big( -|w'|(\alpha-\varepsilon(\delta))+S_{w'}\Phi(x) \Big).\\
\end{align*}
Moreover, we can get
\[
M_{w}(Z,\alpha,\Phi,\mathcal{U},N) \leq e^{-\alpha} \cdot M'_{w}(Z,\alpha-\varepsilon(\delta),\Phi,\frac{1}{2}\gamma(\mathcal{U})),N),
\]
which implies
\[
M(Z,\alpha,\Phi,\mathcal{U},N) \leq e^{-\alpha} \cdot M'(Z,\alpha-\varepsilon(\delta),\Phi,\frac{1}{2}\gamma(\mathcal{U})),N).
\]
Taking the limit $N \rightarrow \infty$ yields
\[
m(Z,\alpha,\Phi,\mathcal{U}) \leq e^{-\alpha} \cdot m'(Z,\alpha-\varepsilon(\delta),\Phi,\frac{1}{2}\gamma(\mathcal{U}))).
\]
Therefore
\[
P_{Z}(G,\Phi,\mathcal{U}) \leq P'_{Z}(G,\Phi,\frac{1}{2}\gamma(\mathcal{U}))+\varepsilon(\delta),
\]
and as $\delta \rightarrow 0$, that is, $\varepsilon(\delta)\rightarrow 0$, $|\mathcal{U}|\rightarrow 0$, we obtain
\[
P_{Z}(G,\Phi) \leq \liminf_{\gamma(\mathcal{U})\rightarrow 0}P'_{Z}(G,\Phi,\frac{1}{2}\gamma(\mathcal{U})).
\]
On the other hand, fix a cover $\mathcal{U}$ of X with $|\mathcal{U}| < \delta$. Given $ w \in F_{m}^{+}, |w|=N$ and $\mathcal{G}_{w}\subset \mathcal{S}(\mathcal{U})$ covering $Z$, we may assume without loss of generality that for every $\textbf{U} \in \mathcal{G}_{w}$, we have $X_{w_{\textbf{U}}}(\textbf{U}) \cap Z \neq \varnothing$. Then for each $\textbf{U} \in \mathcal{G}_{w}$, we can choose $x \in X_{w_{\textbf{U}}}(\textbf{U}) \cap Z$. We observe
\[
X_{w_{\textbf{U}}}(\textbf{U}) \subset B_{w_{\textbf{U}}}(x,\delta).
\]
Using $\mathcal{F}_{w}$ denotes the collection of all $(w_{\textbf{U}},\delta)$-Bowen ball $B_{w_{\textbf{U}}}(x,\delta)$ constructed above and then
\begin{align*}
M_{w}(Z,\alpha,\Phi,\mathcal{U},N)
&=\inf\limits_{\mathcal{G}_{w}}\left\{\sum\limits_{\textbf{U}\in\mathcal{G}_{w}}\exp\left(-\alpha m(\textbf{U})+\sup_{y \in X_{w_{\textbf{U}}}(\textbf{U})} S_{w_{\textbf{U}}}\Phi (y)\right)\right\}\\
& \geq e^{-\alpha}\cdot \inf_{\mathcal{F}_{w}} \left\{\sum\limits_{ B_{w_{\textbf{U}}}(x,\delta)\in \mathcal{F}_{w}}\exp\bigg(-\alpha |w_{\textbf{U}}|+ S_{w_{\textbf{U}}}\Phi(x)\bigg)\right\}\\
&\geq e^{-\alpha}\cdot M'_{w}(Z,\alpha,\Phi,\delta,N).
\end{align*}
It follows
\[
M(Z,\alpha,\Phi,\mathcal{U},N) \geq  e^{-\alpha}\cdot M'(Z,\alpha,\Phi,\delta,N).
\]
Hence
\[
m(Z,\alpha,\Phi,\mathcal{U}) \geq  e^{-\alpha}\cdot m'(Z,\alpha,\Phi,\delta).
\]
Therefore
\[
P_{Z}(G,\Phi,\mathcal{U}) \geq P'_{Z}(G,\Phi,\delta),
\]
and taking the limit as $\delta \rightarrow 0$ gives
\[
P_{Z}(G,\Phi) \geq \limsup_{\delta\rightarrow 0}P'_{Z}(G,\Phi,\delta),
\]
which completes the proof of the first.
The existence of the two other limits can be proved in a similar way.
\end{proof}

\begin{remark}
(1) When $m=1$, i.e., $G=\{f\},~\Phi=\{\varphi\}$, then $P_{Z}(G,\Phi)$ is consistent with the topological pressure using the centre of Bowen ball which is defined by Climenhaga \cite{Climenhaga} for every $Z \subset X$.\\
(2) Similar to Remark \ref{4.2r}, we have
\[
\underline{CP}_{Z}(G,\Phi)=\lim\limits_{\delta \rightarrow 0}\liminf_{N\rightarrow \infty}\frac{\log \Lambda'(Z,\Phi,\delta,N)}{N},
\]
\[
\overline{CP}_{Z}(G,\Phi)=\lim\limits_{\delta \rightarrow 0}\limsup_{N\rightarrow \infty}\frac{\log \Lambda'(Z,\Phi,\delta,N)}{N}.
\]
\end{remark}

\section{\emph{The proof of Theorem 2.1}}

Let $(X,d)$ be compact metric space and $f_{0}, f_{1},\cdots,f_{m-1}$ be continuous maps onto itself. In this section, as the application of the topological pressure introduced in this paper, we give the connection between topological pressure and Hausdorff dimension on some $Z$ in the form of Bowen's equation, which extends the results of Climenhaga \cite{Climenhaga}. Before proving the Theorem 2.1, we give some relevant results.

\begin{proposition}\label{7.1p}
Let $f_{i}: X\longrightarrow X$ be as in Theorem 2.1. Fix $0 < \alpha \leq \beta < \infty$ and $Z \subset \mathcal{A}([\alpha,\beta])$, then

(1) for any $t \in \mathbb{R}$ and $h > 0$, we have
\begin{equation}\label{7.1}
 P_{Z}(G,-t\Phi)-\beta h \leq P_{Z}(G,-(t+h)\Phi)\leq P_{Z}(G,-t\Phi)-\alpha h.
\end{equation}

(2) The equation $P_{Z}(G,-t\Phi)=0$ has unique roots $t^{*}$ and
\[
\frac{h_{Z}(G)}{\beta} \leq t^{*} \leq \frac{h_{Z}(G)}{\alpha}.
\]

(3) If $\alpha=\beta$, then $ t^{*} = \frac{h_{Z}(G)}{\alpha}$.\\
Where $h_{Z}(G)$ is the topological entropy in Ju et al \cite{Ju}, $\Phi=\{\log a_{0},\log a_{1},\cdots, \log a_{m-1}\}$ and $t \Phi=\{t\cdot \log a_{0},t\cdot \log a_{1},\cdots, t\cdot \log a_{m-1}\}$.

\end{proposition}

\begin{proof}
(1) Given arbitrary $\varepsilon >0$ and $k \geq 1$, let
\[
Z_{k}=\Big\{x \in Z: \frac{1}{|w|}S_{w}\Phi(x) \in (\alpha-\varepsilon,\beta+\varepsilon), ~for~any~|w|\geq k\Big\},
\]
and notice that $Z= \cup_{k=1}^{\infty} Z_{k}$. Now fix $t \in \mathbb{R},~h > 0$, $w=i_{1}i_{2}\cdots i_{N} \in F_{m}^{+}$ and $N \geq k$. It follows that for any $\delta >0$, $s \in \mathbb{R}$,
\begin{align*}
M'_{w}(Z_{k},s,-(t+&h)\Phi,\delta,N)\\
&=\inf\limits_{\mathcal{G}_{w} }\left\{\sum\limits_{B_{w'}(x,\delta)\in\mathcal{G}_{w} }\exp \bigg( -s \cdot |w'|-(t+h)S_ {w'} \Phi(x) \bigg)\right\}  \\
&\leq  \inf\limits_{\mathcal{G}_{w} }\left\{\sum\limits_{B_{w'}(x,\delta)\in\mathcal{G}_{w} }\exp \bigg(-s \cdot |w'|-tS_ {w'} \Phi(x)- h |w'| (\alpha-\varepsilon)\bigg)\right\} \\
&=\inf\limits_{\mathcal{G}_{w} }\left\{\sum\limits_{B_{w'}(x,\delta)\in\mathcal{G}_{w} }\exp \bigg(-(s+h(\alpha-\varepsilon))\cdot |w'|-tS_ {w'} \Phi(x)\bigg)\right\} \\
&=M'_{w}(Z_{k},s+h(\alpha-\varepsilon),-t\Phi,\delta,N),
\end{align*}
where $\mathcal{G}_{w} \subset \mathcal{F}$ covers $Z_{k}$.
It follows that
\[
M'(Z_{k},s,-(t+h)\Phi,\delta,N) \leq M'(Z_{k},s+h(\alpha-\varepsilon),-t\Phi,\delta,N)
\]
and then
\[
m'(Z_{k},s,-(t+h)\Phi,\delta) \leq m'(Z_{k},s+h(\alpha-\varepsilon),-t\Phi,\delta).
\]
Then we have
\[
P'_{Z_{k}}(G,-(t+h)\Phi,\delta) \leq P'_{Z_{k}}(G,-t\Phi,\delta)-h(\alpha-\varepsilon).
\]
Letting $\delta \rightarrow 0$, it follows that
\[
P_{Z_{k}}(G,-(t+h)\Phi) \leq P_{Z_{k}}(G,-t\Phi)-h(\alpha-\varepsilon).
\]
Using the Proposition \ref{3.1p} 
and taking the supremum over all $k\geq 1$, we can get
\[
P_{Z}(G,-(t+h)\Phi) \leq P_{Z}(G,-t\Phi)-h(\alpha-\varepsilon);
\]
since $\varepsilon > 0$ is arbitrary, this establishes the right of inequality (\ref{7.1}).

Using the similar computation, we obtain
\begin{align*}
M'_{w}(Z_{k},s,-(t+h)\Phi,\delta,N)
 \geq M'_{w}(Z_{k},s+h(\beta+\varepsilon),-t\Phi,\delta,N).
\end{align*}
It follows that
\begin{align*}
M'(Z_{k},s,-(t+h)\Phi,\delta,N)
 \geq M'(Z_{k},s+h(\beta+\varepsilon),-t\Phi,\delta,N).
\end{align*}
and then
\begin{align*}
m'(Z_{k},s,-(t+h)\Phi,\delta)
 \geq m'(Z_{k},s+h(\beta+\varepsilon),-t\Phi,\delta).
\end{align*}
Hence,
\[
P'_{Z_{k}}(G,-(t+h)\Phi,\delta) \geq P'_{Z_{k}}(G,-t\Phi,\delta)-h(\beta+\varepsilon).
\]
Letting $\delta \rightarrow 0$, it follows that
\[
P'_{Z_{k}}(G,-(t+h)\Phi) \geq P'_{Z_{k}}(G,-t\Phi)-h(\beta+\varepsilon).
\]
Take the supremum over all $k\geq 1$, and by the Proposition \ref{3.1p} we can get
\[
P_{Z}(G,-(t+h)\Phi) \leq P_{Z}(G,-t\Phi)-h(\alpha-\varepsilon);
\]
since $\varepsilon > 0$ is arbitrary, this establishes the left of inequality (\ref{7.1}).
This complete the proof of the inequality (\ref{7.1}).

(2) We observe that the map $t \mapsto P_{Z}(G,-t\Phi)$ is continuous and strictly decreasing. First applied the left of (\ref{7.1}) with $t=0$ and $h=\frac{h_{Z}(G)}{\beta}$, we have
\[
P_{Z}(G,-\frac{h_{Z}(G)}{\beta}\Phi)\geq P_{Z}(G,0)-h_{Z}(G)=0;
\]
and second  applied the right of (\ref{7.1}) with $t=0$ and $h=\frac{h_{Z}(G)}{\alpha}$, we have
\[
P_{Z}(G,-\frac{h_{Z}(G)}{\alpha}\Phi) \leq P_{Z}(G,0)-h_{Z}(G)=0.
\]
Thus it get the desired result by Intermediate Value Theorem.

(3) It follows from (2) immediately.
\end{proof}

Similar to \cite{Climenhaga}, we also have the following proposition. The proof of this proposition is similar to that of \cite{Climenhaga}. Therefore, we omit the proof.

\begin{proposition}
Let $f_{i}: X\longrightarrow X$ be as in Theorem 2.1 and suppose that for any $\omega \in \Sigma_{m}^{+}, ~\lambda_{\omega}(x)$ exists and is positive. Then $x \in \mathcal{B}$ .
\end{proposition}

Before proving Theorem \ref{7.3t}, similar to \cite{Climenhaga}, we give the following two lemmas.

\begin{lemma}\label{7.1l}
Let $f_{i}: X\longrightarrow X$ be as in Theorem 2.1. Then given any $x \in \mathcal{B}$ and $\varepsilon > 0$, there exists $\delta_{0}=\delta_{0}(\varepsilon)>0$ and $\eta=\eta(x,\varepsilon)>0$ such that for each $w \in F_{m}^{+}$ and $0 < \delta < \delta_{0}$,
\begin{equation}\label{7.3}
B(x,\eta\delta e^{-|w|(\lambda_{w}(x)+\varepsilon)}) \subset B_{w}(x,\delta) \subset B(x,\delta e^{-|w|(\lambda_{w}(x)-\varepsilon)}).
\end{equation}
\end{lemma}

\begin{proof}
Since $f_{i}$ is conformal with factor $a_{i}(x)>0$, for each $i \in \{0,1,2\cdots,m-1\}$ we have
\[
\lim_{y\rightarrow x}\frac{d(f_{i}(x),f_{i}(y))}{d(x,y)}=a_{i}(x).
\]
Since $a_{i}(x)>0$ everywhere, we can take logarithms and get
\[
\lim_{y\rightarrow x}(\log d(f_{i}(x),f_{i}(y))-\log d(x,y))=\log a_{i}(x).
\]
It can be extended the continuous function $\zeta_{i}: X \times X\longrightarrow \mathbb{R}$
\begin{equation*}
\zeta_{i}(x,y)=\begin{cases}
\log d(f_{i}(x),f_{i}(y)-\log d(x,y)  \       \ x\neq y, \\
\log a_{i}(x) \       \ x=y.
\end{cases}
\end{equation*}
Because $X \times X$ is compact, $\zeta_{i}$ is uniformly continuous, $i \in \{0,1,\cdots,m-1\}$. Hence for $\varepsilon >0$, there exists $\delta_{0}=\delta_{0}(\varepsilon)>0$
such that for every $0 < \delta < \delta_{0}$ and $(x,y),(x',y') \in X \times X$ with
\[
(d\times d)((x,y),(x',y')):=d(x,x')+d(y,y')<\delta,
\]
we have $|\zeta_{i}(x,y)-\zeta_{i}(x',y')| <\varepsilon$. In particular, for $x,y \in X$ with $d(x,y)<\delta$, we have $(d\times d)((x,y),(x,x))<\delta$.
Therefore,
\[
|\log d(f_{i}(x),f_{i}(y))-\log d(x,y)-\log a_{i}(x)|=|\zeta_{i}(x,y)-\zeta_{i}(x,x)| < \varepsilon,
\]
that is,
\[
\log d(f_{i}(x),f_{i}(y))-\log a_{i}(x)-\varepsilon < \log d(x,y)< \log d(f_{i}(x),f_{i}(y))-\log a_{i}(x)+\varepsilon ,
\]
and taking exponentials, we obtain
\begin{equation}\label{7.4}
d(f_{i}(x),f_{i}(y))e^{-(\log a_{i}(x)+\varepsilon)} <  d(x,y)<  d(f_{i}(x),f_{i}(y))e^{-(\log a_{i}(x)-\varepsilon)} ,
\end{equation}
whenever the middle quantity is less than $\delta$.
Given $w=i_{1}i_{2}\cdots i_{n} \in F_{m}^{+}$. Now we show the second half of (\ref{7.3}). Let $\Phi=\{\log a_{0},\log a_{1},\cdots, \log a_{m-1}\}$. Suppose $y \in B_{w}(x,\delta)$, then $d(f_{w'}(x),f_{w'}(y)) < \delta$ for all $w' \leq \overline{w}$.
Then repeated application of the second inequality in (\ref{7.4}) yields
\begin{align*}
 d(x,y)&< d(f_{i_{1}}(x),f_{i_{1}}(y)) e^{-(\log a_{i_{1}}(x)-\varepsilon)}\\
 &<d(f_{i_{2}i_{1}}(x),f_{i_{2}i_{1}}(y)) e^{-(\log a_{i_{2}}(f_{i_{1}}(x))-\varepsilon)} e^{-(\log a_{i_{1}}(x)-\varepsilon)}\\
 &<\cdots\\
 &<d(f_{\overline{w}}(x),f_{\overline{w}}(y)) e^{-(S_{w}\Phi(x)-|w|\varepsilon )}\\
 &< \delta e^{-|w|(\lambda_{w}(x)-\varepsilon )}.
\end{align*}
Then
\[
B_{w}(x,\delta) \subset B(x,\delta e^{-|w|(\lambda_{w}(x)-\varepsilon)}).
\]

Now we prove the first inclusion in (\ref{7.3}). We observe that if $d(x,y)<\delta$, then by the first inequality in (\ref{7.4}) we get
\[
d(f_{i_1}(x),f_{i_1}(y))<d(x,y) e^{\log a_{i_{1}}(x)+\varepsilon}.
\]
Then if $d(x,y)<\delta e^{-(\log a_{i_{1}}(x)+\varepsilon)}$, we have $d(f_{i_1}(x),f_{i_1}(y))<\delta$ and so
\begin{align*}
d(f_{i_{2}i_{1}}(x),f_{i_{2}i_{1}}(y))&<d(f_{i_1}(x),f_{i_1}(y))e^{\log a_{i_{2}}(f_{i_{1}}x)+\varepsilon}\\
&<d(x,y)e^{\log a_{i_{2}}(f_{i_{1}}x)+\varepsilon} e^{\log a_{i_{1}}(x)+\varepsilon}.
\end{align*}
Continuing in this method, we can obtain that if
\[
d(x,y)<\delta e^{-|w'|(\lambda_{w'}(x)+\varepsilon)}
\]
for each $\overline{w'}\leq \overline{w}$, we have $d(f_{w'}(x),f_{w'}(y))<\delta$ , and hence $y \in B_{w}(x,\delta)$.
Therefore
\begin{equation}\label{7.5}
B(x,\delta \min_{\overline{w'}\leq \overline{w}}e^{-|w'|(\lambda_{w'}(x)+\varepsilon)}) \subset B_{w}(x,\delta),
\end{equation}
which is almost what we wanted. If the minimum was always achieved at $w'=w$, we would be done; however, this may not be the case. Therefore, now we find what $\eta$ should be for any $\overline{w'}\leq \overline{w}$, and we observe that
\begin{align*}
\frac{e^{-|w|(\lambda_{w}(x)+2\varepsilon)}}{e^{-|w'|(\lambda_{w'}(x)+\varepsilon)}}
&=\frac{e^{-S_{w}\Phi(x)-2|w|\varepsilon}}{e^{-S_{w'}\Phi(x)-|w'|\varepsilon}}\\
&=e^{-\big(S_{w}\Phi(x)-S_{w'}\Phi(x)+2|w|\varepsilon-|w'|\varepsilon\big)}\\
&\leq e^{-\big(S_{w}\Phi(x)-S_{w'}\Phi(x)+|w|\varepsilon\big)}.
\end{align*}
Since $x$ satisfies the tempered contraction condition, there exists $\eta=\eta(x,\varepsilon)>0$ such that
\[
\log \eta  < S_{w}\Phi(x)-S_{w'}\Phi(x)+|w|\varepsilon
\]
for arbitrary $w \in F_{m}^{+},~\overline{w'}\leq \overline{w}$, and hence
\[
e^{-\big(S_{w}\Phi(x)-S_{w'}\Phi(x)+|w|\varepsilon\big)}< \frac{1}{\eta}.
\]
Then for such $w,w'$, we have
\[
\eta e^{-|w|(\lambda_{w}(x)+2\varepsilon)}< e^{-|w'|(\lambda_{w'}(x)+\varepsilon)},
\]
and combining with (\ref{7.5})
\[
B(x,\delta \eta e^{-|w|(\lambda_{w}(x)+2\varepsilon)}) \subset B_{w}(x,\delta).
\]
Taking $\delta_{0}=\delta_{0}(\varepsilon/2)$ gives that desired result.
\end{proof}

\begin{remark}
(1) If $x$ has bounded contraction, then $\eta=\eta(x)$ may be chosen independently of $\varepsilon$.\\
(2) Furthermore, if $a_{i}(x) \geq 1, ~i \in \{0,1,\cdots,m-1\}$ for any $x \in X,$ then $\eta=1$ suffices.
\end{remark}

\begin{lemma}\label{7.2l}
Let $f_{i}: X\longrightarrow X$ satisfy the conditions of Theorem 2.1 and fix $Z \subset \mathcal{A}((\alpha,\infty))\bigcap \mathcal{B},$ where $0<\alpha<\infty$. Let $t^{*}$ be the unique real number with $P_{Z}(G,-t^{*} \Phi )=0$, whose existence and uniqueness is guaranteed by Proposition \ref{7.1p}, where $\Phi=\{\log a_{0},\log a_{1},\cdots,~\log a_{m-1}\}$. Then $dim_{H} Z=t^{*}$.
\end{lemma}

\begin{proof}
First we prove $dim_{H} Z \leq t^{*}$. Given $k \geq 1$, and let
\[
Z_{k}=\{x \in Z: \lambda_{w}(x)> \alpha~~for~all~|w|\geq k\}
\]
and then $Z=\bigcup_{k=1}^{\infty} Z_{k}$. Consider $t > t^{*}$, and naturally $P_{Z}(G,-t \Phi)< 0$. Thus there exists $\varepsilon >0$ with $P_{Z}(G,-t \Phi )<-t\varepsilon$ and by Lemma \ref{7.1l} there exists $\delta_{0}=\delta_{0}(\varepsilon)$ such that for all $x \in Z_{k},~0<\delta<\delta_{0}$, fixing $w \in F^{+}_{m}$ and $|w|\geq k$, we have
\begin{equation}\label{7.6}
diam B_{w} (x,\delta)\leq 2 \delta e^{-|w|(\lambda_{w}(x)-\varepsilon)} \leq 2 \delta e^{-|w|(\alpha-\varepsilon)}.
\end{equation}
Given $N\geq k$ and $0< \delta < \delta_{0}$, we have
\begin{align*}
M'_{w}(Z_{k},-t\varepsilon,-t \Phi,\delta,N)
&=\inf\limits_{\mathcal{G}_{w} }\left\{\sum\limits_{B_{w'}(x,\delta)\in\mathcal{G}_{w} }\exp \bigg(-(-t\varepsilon) \cdot |w'|-t S_ {w'} \Phi(x)\bigg)\right\}\\
&=\inf\limits_{\mathcal{G}_{w} }\left\{\sum\limits_{B_{w'}(x,\delta)\in\mathcal{G}_{w} }\exp \bigg(-t \cdot |w'|(\lambda_{w'}(x)-\varepsilon)\bigg)\right\}\\
&\geq \inf\limits_{\mathcal{G}_{w}}\left\{\sum\limits_{B_{w'}(x,\delta)\in\mathcal{G}_{w}} (\frac{1}{2\delta} diam B_{w'}(x,\delta))^{t}\right\} \\
&\geq \inf\limits_{\mathcal{D}(Z_{k},2\delta e^{-N(\alpha-\varepsilon)})} \left\{\sum_{U_{i} \in \mathcal{D}(Z_{k},2\delta e^{-N(\alpha-\varepsilon)})} (\frac{1}{2\delta} diam U_{i})^{t}\right\} \\
&=(2\delta)^{-t} m_{H} (Z_{k},t,2\delta e^{-N(\alpha-\varepsilon)}), \\
\end{align*}
where $\mathcal{G}_{w} \subset \mathcal{F}$ covers $Z_{k}$ and $\mathcal{D}(Z_{k},2\delta e^{-N(\alpha-\varepsilon)})$ denotes the collection of open covers $\{ U_{i}\}$ of $Z_{k}$ for which $diam U_{i} < 2\delta e^{-N(\alpha-\varepsilon)}$ for all $i$.\\
It follows that
\begin{align*}
 M'(Z_{k},-t\varepsilon,-t \Phi,\delta,N)&=\frac{1}{m^{N}} \sum_{|w|=N} M'_{w}(Z_{k},-t\varepsilon,-t \Phi,\delta,N)\\
 &\geq (2\delta)^{-t} m_{H} (Z_{k},t,2\delta e^{-N(\alpha-\varepsilon)}) .
\end{align*}
Taking the limit as $N \rightarrow \infty$ gives
\begin{equation}\label{7.7}
m'(Z_{k},-t\varepsilon,-t \Phi,\delta) \geq (2\delta)^{-t} m_{H}(Z_{k},t).
\end{equation}
Thus
\[
-t\varepsilon > P_{Z}(G,-t \Phi )\geq P_{Z_k}(G,-t \Phi)=\lim_{\delta \rightarrow 0} P'_{Z_k}(G,-t \Phi, \delta),
\]
and for sufficiently small $\delta >0$, we have $-t \varepsilon > P'_{Z_k}(G,-t \Phi,\delta )$.
Hence $ H^{t}(Z_{k})=0$ by (\ref{7.7}), which implies $t \geq dim_{H}(Z_{k})$. Then taking the union over all $k$ gives $dim_{H}(Z) \leq  t $ for all $t>t^{*}$. Therefore $dim_{H}(Z) \leq t^{*}$.

Now we prove the other inequality, $dim_{H}(Z) \geq t^{*}$. Fix $t <t^{*}$, and then we show that $t\leq dim_{H}(Z)$. Suppose $t>0$, and by the Proposition \ref{7.1p} $t^{*}$ is the unique real number such that $P_{Z}(G,-t^{*} \Phi)=0$. Since the pressure function $P_{Z}(G,-t\Phi)$ is decreasing, we have $P_{Z}(G,-t \Phi)>0$. Hence we can choose $\varepsilon >0$ such that
\[
P_{Z}(G,-t\Phi )>t\varepsilon > 0.
\]
Set $\delta_{0}=\delta_{0}(\varepsilon)$ be as in Lemma \ref{7.1}. Given $k\in \mathbb{N},~k\geq 1$, consider the set
\[
Z_{k}=\{x \in Z: (\ref{7.3})~holds~with~\eta=e^{-k}~for~all~w \in F_{m}^{+} ~and ~0<\delta<\delta_{0}\}.
\]
Obviously, $Z=\bigcup\limits_{k=1}^{\infty} Z_{k}$ and hence $P_{Z}(G,-t\Phi)=\sup\limits_{k\geq 1} P_{Z_k}(G,-t \Phi)$. Then there exists $k\in \mathbb{N}$ with $t\varepsilon<P_{Z_k}(G,-t\Phi)$, and we fix $0<\delta<\delta_{0}$ such that
\begin{equation}\label{7.8}
t\varepsilon<P'_{Z_k}(G,-t\Phi,\delta).
\end{equation}
Let $\beta=\max_{i} \{\sup_{x \in X}\log a_{i}(x)\}<\infty$. For any $w=i_{1}i_{2}\cdots i_{n} \in F_{m}^{+},~x\in X$, denote $s_{w}(x)=e^{-k} \delta e^{-|w|(\lambda_{w}(x)+\varepsilon)}$, and notice that for any $i_{n+1} \in \{0,1,\cdots,m-1\}$
\begin{align*}
\frac{s_{w}(x)}{s_{wi_{n+1}}(x)}=\frac{e^{-S_{w} \Phi(x)-|w|\varepsilon}}{e^{-S_{wi_{n+1}} \Phi(x)-(|w|+1)\varepsilon}}=e^{\log a_{i_{n+1}}(f_{\overline{w}}(x))+\varepsilon} \leq e^{\beta+\varepsilon}.
\end{align*}
Moreover, given $x \in Z_{k}$ and $r>0$ enough small and for any $\omega=i_{1}i_{2}\cdots \in \Sigma_{m}^{+}$, there exists $n=n(x,r,\omega) \in \mathbb{N}$ such that for $w= i_{1}i_{2} \cdots i_{n}=\omega|_{[0,n-1]}\in F_{m}^{+}(n)$, we have
\begin{equation}\label{7.9}
s_{w}(x)e^{-(\beta+\varepsilon)}\leq s_{wi_{n+1}}(x) \leq r \leq s_{w}(x)=e^{-k} \delta e^{-|w|(\lambda_{w}(x)+\varepsilon)}.
\end{equation}

Moreover, for every $w' \in F_{m}^{+}$ and $x \in X$, we have $\lambda_{w'}(x) \leq \beta$ and thus $s_{w'}(x) \geq \delta e^{-\big(k+|w'|(\beta+\varepsilon)\big)}$.
It yields from (\ref{7.9}) that for $w=i_{1}i_{2}\cdots i_{n}=\omega|_{[0,n-1]}$, where $n=n(x,r,\omega)$, we have
\[
\delta e^{-\big(k+(n+1)(\beta+\varepsilon)\big)} \leq r,
\]
and thus
\[
n \geq \frac{-\log r +\log \delta -k}{\beta+\varepsilon}-1.
\]
Writing $N:=N(r,\delta)=\frac{-\log r +\log \delta -k}{\beta+\varepsilon}-1$ and observe that for every fixed $\delta>0$, we have
$\lim_{r\rightarrow 0}N(r,\delta)=\infty$.

By Lemma \ref{7.1l} we have for $w=i_{1}i_{2}\cdots i_{n}=\omega|_{[0,n-1]},~n=n(x,r,\omega) \in \mathbb{N}$,
\[
B(x,r)\subset  B_{w}(x,\delta),
\]
then given any $\{B(x_{i},r_{i})\}$ with $Z_{k} \subset \bigcup B(x_{i},r_{i})$, we can get $Z_{k} \subset \bigcup  B_{w_{i}}(x_{i},\delta)$, where $w_{i}=\omega|_{[0,n_{i}-1]},~n_{i}=n_{i}(x_{i},r_{i},\omega)$ satisfies (\ref{7.9}).

 Thus for all $r>0,~0<\delta <\delta_{0}$ and fixed $w=\omega|_{[0,N-1]} \in F^{+}_{m}(N)$, we can get
\begin{align*}
m^{b}_{H}(Z_{k},t,r)
&=\inf_{\mathcal{D}^{b}(Z_{k},r)}\left\{\sum\limits_{B(x_{i},r_{i})\in \mathcal{D}^{b}(Z_{k},r)} (2r_{i})^{t}\right\}\\
&\geq \inf\limits_{\mathcal{G}'_{w} }\left\{\sum\limits_{B_{w_{i}}(x_{i},\delta)\in\mathcal{G}'_{w}}\bigg(2 e^{-(\beta+\varepsilon)}s_{w_{i}}(x_{i})\bigg)^{t}\right\}\\
&=(2\delta)^{t}e^{-t(k+\beta+\varepsilon)} \inf\limits_{\mathcal{G}'_{w} }\left\{\sum\limits_{B_{w_{i}}(x_{i},\delta)\in\mathcal{G}'_{w}}  e^{-|w_{i}|t(\lambda_{w_{i}}(x_{i})+\varepsilon)}\right\} \\
&=(2\delta)^{t}e^{-t(k+\beta+\varepsilon)}  \inf\limits_{\mathcal{G}'_{w} }\left\{\sum\limits_{B_{w_{i}}(x_{i},\delta)\in\mathcal{G}'_{w}}  \exp \bigg(-|w_{i}|t\varepsilon-tS_{w_{i}} \Phi(x_{i})\bigg)\right\} \\
&\geq (2\delta)^{t}e^{-t(k+\beta+\varepsilon)}  \inf\limits_{\mathcal{G}_{w} }\left\{\sum\limits_{B_{w'}(x,\delta)\in\mathcal{G}_{w}}  \exp \bigg(-|w'|t\varepsilon-tS_{w'} \Phi(x)\bigg)\right\} \\
&=(2\delta)^{t}e^{-t(k+\beta+\varepsilon)} M'_{w}(Z_{k},t\varepsilon,-t\Phi,\delta,N) , \\
\end{align*}
where $\mathcal{D}^{b}(Z_{k},r)$ denotes the collection of countable open balls covers $\{B(x_{i},r_{i})\}_{i=1}^{\infty}$ of $Z_{k}$ for which $r_{i} < r$ for all $i$,
$\mathcal{G}'_{w}$ denotes the collection of countable Bowen balls covers $\{B_{w_{i}}(x_{i},\delta)\}_{i=1}^{\infty}$ of $Z_{k}$ for which $~\overline{w}\leq \overline{w_{i}},~w=\omega|_{[0,N-1]}, w_{i}=\omega|_{[0,n_{i}-1]}$  for all $i$
and $\mathcal{G}_{w}$ denotes the collection of countable Bowen balls covers $\{B_{w'}(x,\delta)\}$ of $Z_{k}$ for which $~\overline{w}\leq \overline{w'},~w=\omega|_{[0,N-1]}$.\\
It follows that
\[
m^{b}_{H}(Z_{k},t,r)\geq (2\delta)^{t}e^{-t(k+\beta+\varepsilon)} M'(Z_{k},t\varepsilon,-t\Phi,\delta,N).
\]
Taking the limit as $r\rightarrow 0$, we can obtain the quantity on the right goes to $\infty$ by (\ref{7.8}), and thus we have $m^{b}_{H}(Z_{k},t)=\infty$.
Therefore,
\[
dim_{H}Z \geq dim_{H}Z_{k}\geq t,
\]
and since $t< t^{*}$ was arbitrary, this establishes the Lemma.
\end{proof}

\begin{proof}[ Proof of Theorem 2.1]
Consider a decreasing sequence of positive numbers $\alpha_{k}$ which converge to $0$, and set $Z_{k} \subset \mathcal{A}((\alpha_{k},\infty))\bigcap Z$, so that Lemma \ref{7.2l} applies to $Z_{k}$ and we have $Z=\bigcup_{k=1}^{\infty} Z_{k}$. Let $t_{k}$ be the unique real number with
\[
P_{Z_{k}}(G,-t_{k} \Phi )=0
\]
for every $k$ and whose existence and uniqueness is guaranteed by Proposition \ref{7.1p}. Thus by Lemma \ref{7.2l} we get
\[
dim_{H} Z_{k}=t_{k}.
\]
Denote $t^{*}=\sup_{k} t_{k}$, then $dim_{H} Z=t^{*}$. Then it remains to prove that
\begin{equation}\label{7.10}
t^{*}=\sup\{ t\geq 0: P_{Z}(G,-t \Phi)>0\}.
\end{equation}
Given $t\geq 0$, we have
\[
P_{Z}(G,-t \Phi )=\sup_{k}P_{Z_{k}}(G,-t \Phi ).
\]

First for any $t < t^{*}$, there exists $t_{k}$ with $t<t_{k}$ and hence $P_{Z_{k}}(G,-t \Phi )>0$. Then $t \in \{ t\geq 0: P_{Z}(G,-t \Phi )>0\}$ and thus $t \leq \sup\{ t\geq 0: P_{Z}(G,-t\Phi )>0\}$. Therefore, $t^{*} \leq \sup\{ t\geq 0: P_{Z}(G,-t\Phi )>0\}$.

Next for arbitrary $t < \sup\{ t\geq 0: P_{Z}(G,-t \Phi)>0\}$, there exists $t_{j}>t$ with $P_{Z}(G,-t_{j}\Phi )>0$. Then there exists $Z_{k}$ such that $P_{Z_{k}}(G,-t_{j} \Phi )>0$ and thus $t_{j} < t_{k}$. It follows that $t<t_{j}<t_{k}<t^{*}$.
So  $\sup\{ t\geq 0: P_{Z}(G,-t\Phi )>0\} \leq t^{*}$.
This establishes (\ref{7.10}).

Finally, it follows from (\ref{7.10}) and continuity of $t \mapsto P_{Z}(G, -t \Phi)$ that $P_{Z}(G, -t^{*} \Phi)=0$. If $Z \subset \mathcal{A}((\alpha,\infty))$ for some $\alpha >0$, then Proposition \ref{7.1p} guarantees that $t^{*}$ is in fact the unique root of Bowen's equation.
\end{proof}

\begin{example}
Let $X=[0,1]$ be a closed interval on $\mathbb{R}$. And $G$ is a free semigroup generated by a series of the Manneville-Pomeau maps (for the study of Manneville-Pomeau map, see, for example \cite{Pomeau,Takens}), that is,
\[
f_{i}: X\rightarrow X, f_{i}(x)=x+x^{1+s_{i}}~(mod~1),i=0,1,\cdots,m-1
\]
and $0<s_{0}<s_{1}<\cdots<s_{m-1}<1.$ Then $f_{i}~(i=0,1,\cdots,m-1)$ is continuous conformal and has no critical points and singularities, and $a_{i}(x)\geq 1$ for any $x \in X$. The Manneville-Pomeau map is non-uniformly hyperbolic transformation having the most benign type of non-hyperbolicity: an indifferent fixed point at 0, i.e., $f_{i}(0)=0$, and $a_{i}(0)=1$, and exhibits \emph{intermittent} behavior. Moreover, we can verify $X$, $G=\{ f_{0},f_{1},\cdots,f_{m-1}\}$ and $\Phi=\{\log a_{0},\log a_{1},\cdots,\log a_{m-1}\}$ satisfy the conditions of Theorem \ref{7.3t}. What's more, for any $\omega \in \Sigma^{+}_{m}$
\[
\overline{\lambda}_{\omega}(0)=\underline{\lambda}_{\omega}(0)=\lambda_{\omega}(0)=0
\]
and
\[
\underline{\lambda}_{\omega}(x)>0, ~for ~any~x \in X-\{0\}
\]
thus $\mathcal{B}=X$ and $\mathcal{A}((0,\infty))=X-\{0\}$. Therefore, for any $Z \subset \mathcal{A}((0,\infty))\bigcap \mathcal{B}$, by the Theorem \ref{7.3t} we get
\begin{align*}
dim_{H} Z=t^{*}&=\sup\{ t\geq 0: P_{Z}(G,-t \Phi )>0\}\\
&=\inf \{ t\geq 0: P_{Z}(G,-t \Phi ) \leq 0\}.
\end{align*}
\end{example}



\section{\emph{The proof of Theorem 2.2}}

\begin{proof}[ Proof of Theorem 2.2]
(1) Fix an $\varepsilon> 0$ and for each $ k\geq 1 $. Consider the set
\[
 Z_k=\left\{ x \in Z : \liminf_{n\rightarrow \infty}-\frac{1}{n}\max_{|w|=n}\{\log\mu(B_{w}(x,r))-S_{w}\Phi(x)\}> s-\varepsilon~{\rm for~all}~r \in (0,1/k)\right\}.
 \]
Since $\underline{P}(G, \Phi, x)\ge s$ for all $x\in Z$, the sequence $ \{ Z_k\}_{k=1}^{\infty} $ increases to $ Z $. So by the continuity of the measure, we have
 \[
\lim\limits_{k\rightarrow\infty}\mu(Z_k)=\mu(Z)>0.
 \]
 Then select an integer $ k_0 \geq 1 $ with $ \mu(Z_{k_0})>\frac{1}{2}\mu(Z)$. For each $ N\geq1$, put
\begin{align*}
Z_{k_0,N}=\bigg\{ x \in Z_{k_0} :-\frac{1}{n}\max_{|w|=n}\{\log&\mu(B_{w}(x,r))
-S_{w}\Phi(x)\}> s-\varepsilon \\
&~{\rm for~all}~n \ge N~{\rm and}~r \in (0,1/k_0)\bigg\}.
\end{align*}
Since the sequence $ {\{ Z_{k_0,N}\}}_{N=1}^{\infty} $ increases to $ Z_{k_0}$, we can pick a $ N^{*}\geq 1 $ such that $\mu( Z_{k_0,N^{*}})>\frac{1}{2}\mu( Z_{k_0})$. Write $Z^{*}=Z_{k_0,N^{*}}$ and $ r^{*}=\frac{1}{k_0} $. Then $ \mu(Z^{*}) >0$, for all $x \in Z^{*},~0<r \leq r^{*}$ and $n \ge N^{*}$,
 \begin{equation}
 \max_{|w|=n}\{\log\mu\big (B_{w}(x,r)\big)-S_{w}\Phi(x)\}<  -(s-\varepsilon)\cdot n.
 \end{equation}
 For any $N \geq N^{*}, w \in F_{m}^{+} $ and $|w|= N$, take a cover of $Z^{*}$
 \[
 \mathcal{F}_{w}= \left\{B_{\omega_i|_{[0,n_{i}-1]}}(y_i,\frac{r}{2}) : \omega_i \in \Sigma^+_m~{\rm and}~\omega_i|_{[0,N-1]}=w \right\}
 \]
 such that
\[
 Z^{*}\cap B_{\omega_i|_{[0,n_{i}-1]}}(y_i,\frac{r}{2})\neq\emptyset, n_{i} \geq N ~ {\rm for~all}~ i\geq1 ~{\rm and}~ 0<r \leq r^{*}.
\]
For each $ i $, there exists an $x_i\in Z^{*}\cap B_{\omega_i|_{[0,n_{i}-1]}}(y_i,\frac{r}{2})$. By the triangle inequality
 \[B_{\omega_i|_{[0,n_{i}-1]}}(y_i,\frac{r}{2}) \subset B_{\omega_i|_{[0,n_{i}-1]}}(x_i,r).\]
In combination with (7.1), we can get
\[
\sum_{i\geq1} \exp \Big(-(s-\varepsilon)\cdot n_i +S_{w}\Phi(x_{i})\Big)
 \ge \sum_{i\geq1}  \mu  \big(B_{\omega_i|_{[0,n_{i}-1]}}(x_i,r) \big)
 \ge  \mu(Z^{*})>0.
\]
 Therefore, $M'_{w}(Z^{*},s-\varepsilon,\Phi,r,N) \ge \mu(Z^{*})>0$ for all $N\ge N^{*}$, and we obtain
 \[
 M'(Z^{*},s-\varepsilon,\Phi,r,N)=\frac{1}{m^N}\sum \limits_{|w|=N}M'_{w}(Z^{*},s-\varepsilon,\Phi,r,N) \ge \mu(Z^{*})>0,
 \]
 and consequently
 \[
 m'(Z^{*},s-\varepsilon,\Phi,r)=\lim\limits_{N\rightarrow \infty}M'(Z^{*},s-\varepsilon,\Phi,r,N)>0,
 \]
which in turn implies that $P'_{Z^{*}}(G,\Phi,r)\ge s-\varepsilon$.
Then we have $P_{Z^{*}}(G,\Phi) \ge s-\varepsilon$ by letting $ r \rightarrow 0$. It follows that
 $P_{Z}(G,\Phi)\ge P_{Z^{*}}(G,\Phi)\ge s-\varepsilon$ and hence $P_{Z}(G,\Phi) \ge s$ since $\varepsilon>0$ is arbitrary.

 (2) In order to prove the second result, we need to use the following Lemma.
 \begin{lemma}[\cite{Ju},\cite{Mj} ]
Let $r > 0$ and $\mathcal{B}(r) = \{B_w(x,r): x \in X, w \in F_{m}^{+}\}$. For any family $\mathcal{F}\subset \mathcal{B}(r)$, there exists a (not
necessarily countable) subfamily $\mathcal{G} \subset \mathcal{F}$ consisting of disjoint balls such that
\[
\bigcup_{B \in \mathcal{F}}B \subset \bigcup_{B_w(x,r) \in \mathcal{G}} B_w(x,3r).
\]
\end{lemma}
Since $ \overline{P}(G,\Phi,x)\le s$ for all $ x\in Z $, then for any $ \omega \in \Sigma_{m}^+ $ and $ x \in Z $,
\[
\lim\limits_{r\rightarrow 0}\liminf_{n\rightarrow \infty}-\frac{1}{n}\{\log\mu \big(B_{\omega|_{[0,n-1]}}(x,r) \big)-S_{\omega|_{[0,n-1]}}\Phi(x)\}\leq  \overline{P}(G,\Phi,x)\leq s.
\]
For any $ N \geq 1 $ and $ w=i_1i_2 \ldots i_N \in F_{m}^+$. Fix $\varepsilon> 0$, and put
\begin{align*}
 Z_k=\bigg\{ x \in Z  : &\liminf_{n\rightarrow \infty}-\frac{1}{n}\{\log\mu \big(B_{\omega|_{[0,n-1]}}(x,r) \big)-S_{\omega|_{[0,n-1]}}\Phi(x)\}< s+\varepsilon \\
 &~{\rm for~all}~r \in (0,1/k), \omega \in \Sigma^+_m~{\rm and}~\omega|_{[0,N-1]}=w \bigg\},
 \end{align*}
then we have $Z=\cup_{k \ge 1}Z_k$. Now fix $k \ge 1$ and $ 0<r<\frac{1}{3k} $. For each $x \in Z_k,$ we take $ \omega_{x} \in \Sigma^+_m $ such that $ \omega_{x}|_{[0,N-1]}=w $,    there exists a strictly increasing sequence $\{n_{j}(x)\}_{j=1}^{\infty}$ such that
 \[
\log\mu  \big(B_{\omega_{x}|_{[0,n_{j}(x)-1]}}(x,r) \big)-S_{\omega_{x}|_{[0,n_{j}(x)-1]}}\Phi(x)\ge  -(s+\varepsilon)\cdot n_{j}(x)
\]
for all $j \ge 1$.
So, the set $Z_k$  is contained in the union of the sets in the family
\[
\mathcal{F}_{w}=\left\{B_{\omega_x|_{[0,n_{j}(x)-1]}}(x,r): x \in  Z_k,  \omega_x \in \Sigma^+_m, \omega_x|_{[0,N-1]}=w~{\rm and}~n_{j}(x)\ge N\right\}.
\]
By Lemma 7.1, there exists a sub family $\mathcal{G}_{w}=\{B_{\omega_{x_i}|_{[0,n_{i}-1]}}(x_i,r)\}_{i\in I} \subset \mathcal{F}_{w}$ consisting of disjoint balls such that for all $i \in I$,
\[
Z_k \subset  \bigcup_{i \in I} B_{\omega_{x_i}|_{[0,n_{i}-1]}}(x_i,3r),
\]
and
\[
\mu \big(B_{\omega_{x_{i}}|_{[0,n_{i}-1]}}(x_i,r) \big)\ge \exp  \Big(-(s+\varepsilon)\cdot n_{i} +S_{\omega_{x_{i}}|_{[0,n_{i}-1]}}\Phi(x_{i})\Big).
\]
The index set $I$ is at most countable since $\mu$ is a probability measure and $\mathcal{G}_{w}$ is a disjointed family of sets, each of
which has positive $\mu$-measure. Therefore,
\begin{align*}
M'_{w}(Z_k,s+\varepsilon,\Phi,3r,N)&\le \sum_{i\in I}\exp  \big(-(s+\varepsilon)\cdot n_{i} +S_{\omega_{x_i}|_{[0,n_{i}-1]}}\Phi(x_{i})\big)\\
&\le \sum_{i\in I}\mu  \big(B_{\omega_{x_i}|_{[0,n_{i}-1]}}(x_i,r) \big)\le 1,
\end{align*}
where the disjointness of $\{B_{\omega_{x_i}|_{[0,n_{i}-1]}}(x_i,r)\}_{i\in I}$ is used in the last inequality. It follows that
 \[
 M'(Z_k,s+\varepsilon,\Phi,3r,N)=\frac{1}{m^N}\sum \limits_{|w|=N}M'_{w}(Z_k,s+\varepsilon,\Phi,3r,N)\le 1
 \]
 and consequently
 \[
 m'(Z_k,s+\varepsilon,\Phi,3r)=\lim\limits_{N\rightarrow \infty}M'(Z_k,s+\varepsilon,\Phi,3r,N) \le 1.
 \]
 which in turn implies that $P'_{Z_k}(G,\Phi,3r) \le s+\varepsilon$ for any $0<r<\frac{1}{3k}$. Letting $r \rightarrow 0$ yields
 \[
 P_{Z_k}(G,\Phi) \le s+\varepsilon~{\rm for~any}~k\ge 1.
 \]
By Proposition \ref{3.1p}(3),
 \[
P_{Z}(G,\Phi)=P_{\cup_{k=1}^{\infty}Z_k}(G,\Phi) = \sup_{k\ge 1} \{P_{Z_k}(G,\Phi)\} \le s+\varepsilon.
 \]
Since $\varepsilon > 0$ is arbitrary, then we can get $P_{Z}(G,\Phi)\le s$ .
 \end{proof}

\begin{corollary}\label{6.1t}
Let $\mu$ denote a Borel probability measure on $X$, $Z$ be a Borel subset of $X$ and $ w \in F^{+}_{m},~|w|=n$. Consider the following quantities:
\[
\overline{P}:=\sup_{x \in Z}\lim_{\delta \rightarrow 0} \limsup_{n \rightarrow \infty} -\frac{1}{n} \min_{|w|=n}\left\{ \log \mu (B_{w}(x,\delta))-S_{w}\Phi (x)\right\},
\]
\[
\underline{P}:=\inf_{x \in Z}\lim_{\delta \rightarrow 0} \liminf_{n \rightarrow \infty} -\frac{1}{n} \max_{|w|=n}\left\{ \log \mu (B_{w}(x,\delta))-S_{w}\Phi (x)\right\}.
\]
Then $P_{Z}(G,\Phi) \leq \overline{P}$. If in addition $\mu(Z) > 0$, we have $P_{Z}(G,\Phi) \geq \underline{P}$.
\end{corollary}

\begin{proof}
(1) It is easy to verify that $\overline{P}(G,\Phi,x) \leq \overline{P}$ for any $x \in Z$. Hence, by Theorem 2.2(2), we can get $P_{Z}(G,\Phi) \leq \overline{P}$.

(2) It is obvious to get $\underline{P}(G,\Phi,x) \geq \underline{P}$ for any $x \in Z$. And $\mu(Z) > 0$, then we have $P_{Z}(G,\Phi) \geq \underline{P}$ by Theorem 2.2(1).
\end{proof}

\begin{remark}
When $m=1$, i.e., $G=\{f\},~\Phi=\{\varphi\}$, the Corollary 7.2 coincides with the results that Climenhaga proved in \cite{Climenhaga1}.
\end{remark}

\section{\emph{The proof of Theorem 2.3}}
Let $X$ be a compact metric space with metric $d$, suppose a free semigroup with $m$ generators acting on $X$, the generators are  continuous transformations $G=\{ f_0, f_1, \ldots, f_{m-1}\}$ of $X$ and $\Phi=\{\varphi\}$, i.e., $\varphi_{0}=\varphi_{1}=\cdots=\varphi_{m-1}=\varphi \in C(X, \mathbb{R})$. For $w=i_{1}i_{2}\cdots i_{n} \in F_{m}^{+},$  denote $(S_{w} \Phi)(x):=(S_{w} \varphi)(x)$.

For any subset $Z\subset X$, $w \in F_m^+$ and $\varepsilon >0$, a subset $E \subset X$ is said to be a $(w, \varepsilon, Z, f_0, \ldots, f_{m-1})$-spanning set of $Z$, if for any $x \in Z$, there exists $y \in E$ such that $d_w(x,y) < \varepsilon$. A subset $F \subset Z$ is said to be a $(w, \varepsilon, Z, f_0, \ldots, f_{m-1})$-separated set of $Z$, if $x,y \in F$, $x\neq y$ implies $d_w(x,y) \geq \varepsilon$.

For any $w \in F_{m}^{+}$, $|w|=n$ and $\varepsilon > 0$, put
\begin{align*}
Q_{w}&(Z,G,\varphi,\varepsilon,n)\\
&=\inf \left\{\sum_{x \in E} e^{(S_{w}\varphi)(x)}: E~is~a~(w, \varepsilon, Z, f_0, \ldots, f_{m-1})-spanning~set~of~Z\right\},
\end{align*}
\[
 Q(Z,G,\varphi,\varepsilon,n)=\frac{1}{m^{n}} \sum_{|w|=n} Q_{w}(Z,G,\varphi,\varepsilon,n).
\]
Set
\[
Q(Z,G,\varphi,\varepsilon)=\limsup_{n \rightarrow \infty} \frac{1}{n} \log Q(Z,G,\varphi,\varepsilon,n)
\]

Similarly, for any $w \in F_{m}^{+}$, $|w|=n$ and $\varepsilon > 0$, set
\begin{align*}
P_{w}&(Z,G,\varphi,\varepsilon,n)\\
&=\sup \left\{\sum_{x \in F} e^{(S_{w}\varphi)(x)}: F~is~a~(w, \varepsilon, Z, f_0, \ldots, f_{m-1})-separated~set~of~Z\right\},
\end{align*}
\[
P(Z,G,\varphi,\varepsilon,n)=\frac{1}{m^{n}} \sum_{|w|=n} P_{w}(Z,G,\varphi,\varepsilon,n).
\]
Set
\[
P(Z,G,\varphi,\varepsilon)=\limsup_{n \rightarrow \infty} \frac{1}{n} \log P(Z,G,\varphi,\varepsilon,n).
\]

For any $ w \in F^+_m, |w|=n$, since
\[
\Lambda'_{w}(Z,\varphi,\varepsilon,n)=\inf\limits_{\mathcal{G}_{w}}\left\{\sum\limits_{B_{w}(x,\varepsilon) \in \mathcal{G}_{w}}\exp\bigg((S_{w}\varphi) (x)\bigg)\right\}
=Q_{w}(Z,G,\varphi,\varepsilon,n)
\]
then
$\Lambda'(Z,\varphi,\varepsilon,n)=Q(Z,G,\varphi,\varepsilon,n)$.
Therefore, by Remark 5.2(2) we can obtain
\begin{align*}
\overline{CP}_{Z}(G,\varphi)&=\lim\limits_{\varepsilon \rightarrow 0}\limsup_{n \rightarrow \infty} \frac{1}{n} \log Q(Z,G,\varphi,\varepsilon,n).\\
\end{align*}
 If $\delta=\sup \{|\varphi(x)-\varphi(y)|: d(x,y) \leq \frac{\varepsilon}{2}\}$, similar to the Lin et al ([17, Remark 3.4(5)]), we have
\[
Q(Z,G,\varphi,\varepsilon,n) \leq  P(Z,G,\varphi,\varepsilon,n) \leq e^{n\delta} Q(Z,G,\varphi,\frac{\varepsilon}{2},n).
\]
Moreover, we have
\[
\overline{CP}_{Z}(G,\varphi)=\lim\limits_{\varepsilon \rightarrow 0}\limsup_{n \rightarrow \infty} \frac{1}{n} \log P(Z,G,\varphi,\varepsilon,n).
\]

We assign the following skew-product transformation. It's base is $\Sigma_m$, it's fiber is X, and the maps $F: \Sigma_m \times X \to \Sigma_m \times X$ and $g: \Sigma_m \times X \to \mathbb{R}$ are defined by the formula
\[
F(\omega, x) = (\sigma_m \omega, f_{\omega_0}(x))
\]
and
\[
g(\omega, x) =c+\varphi(x),
\]
where $\omega = (\ldots, \omega_{-1}, \omega_0, \omega_1, \ldots) \in  \Sigma_m$ and $c$ is a constant number.
Here $f_{\omega_0}$ stands for $f_0$ if $\omega_0=0$, and for $f_1$ if $\omega_0=1$, and so on and
$\varphi \in C(X,\mathbb{R})$. Obviously, $g \in  C(\Sigma_m \times X ,\mathbb{R})$. Then
\begin{align*}
F^{n}(\omega, x)&= (\sigma^{n}_m \omega, f_{\omega_{n-1}}f_{\omega_{n-2}} \cdots f_{\omega_0}(x))\\
&=(\sigma^{n}_m \omega, f_{\overline{\omega|_{[0,n-1]}}}(x)).
\end{align*}
Moreover, $S_{n}g(\omega, x)=nc+S_{\overline{\omega|_{[0,n-1]}}}\varphi(x)$.
Let $\omega=(\ldots,\omega_{-1},\omega_0,\omega_1,\ldots)\in \Sigma_m$,
and the metric $D$ on $\Sigma_m \times X$ is defined as
\[
D\big((\omega,x), (\omega', x')\big) = \max\big(d'(\omega, \omega'), d(x, x')\big).
\]

For $ F: \Sigma_m \times X \to \Sigma_m \times X$. Let $G=\{F\}$, we can get a C-P structure on $\Sigma_m \times X$. For any set $ Z \subset X $, from Pesin \cite{Pesin}, we denote $\overline{CP}_{ \Sigma_m \times Z}(F,g)$ the upper capacity topological pressure of $ F $ on the set  $\Sigma_m \times Z$.
Our purpose is to find the relationship between the upper capacity topological pressure $\overline{CP}_{ \Sigma_m \times Z}(F,g)$ of the skew-product transformation $F$ and the upper capacity topological pressure $\overline{CP}_{Z}(G,\varphi)$ of a free semigroup action generated by $ G=\{f_0, \ldots, f_{m-1}\} $ on $Z$.

To prove this theorem, we give the following two lemmas. The proofs of these two lemmas are similar to that of Lin et al \cite{Lin}.
Therefore, we omit the proof.

\begin{lemma}\label{lem:skewleft}
For any subset $Z$ of $X$, $n \geq 1$ and $0 < \varepsilon < \frac{1}{2}$, we have
\begin{equation*}
P(\Sigma_m \times Z,F,g,\varepsilon,n) \geq e^{nc} m^{n} P(Z,G, \varphi, \varepsilon,n).
\end{equation*}
\end{lemma}

\begin{lemma}\label{lem:skewright}
For any subset $Z$ of $X$, $n \geq 1$ and $\varepsilon >0$, we have
\begin{equation*}\label{eq:rightSkew}
Q(\Sigma_m \times Z,F,g, \varepsilon,n) \leq K(\varepsilon) e^{nc} m^{n} Q(Z,G,\varphi, \varepsilon,n),
\end{equation*}
where $K(\varepsilon)$ is a positive constant that depends only on $\varepsilon$.
\end{lemma}

\begin{proof}[ Proof of Theorem 2.3]
From Lemma \ref{lem:skewleft} we have for any  subset $Z$ of $X$,
\[
P(\Sigma_m \times Z,F,g,\varepsilon,n) \geq e^{nc} m^{n} P(Z,G, \varphi, \varepsilon,n).
\]
whence, taking logarithms and limits, we obtain that
\[
\overline{CP}_{\Sigma_m \times Z}(F,g) \geq \log m + \overline{CP}_{Z}(G,\varphi)+c.
\]
In the same way, from Lemma \ref{lem:skewright}, we have
\[
Q(\Sigma_m \times Z,F,g, \varepsilon,n) \leq K(\varepsilon) e^{nc} m^{n} Q(Z,G,\varphi, \varepsilon,n),
\]
whence
\[
\overline{CP}_{\Sigma_m \times Z}(F,g)\leq \log m + \overline{CP}_{Z}(G,\varphi)+c.
\]
Thus the proof is complete.
\end{proof}

{\bf Acknowledgement.}  The work was supported by National Natural Science Foundation of China (grant no.11771149, 11671149) and Guangdong Natural Science Foundation 2018B0303110005.

\bibliographystyle{amsplain}

\end{document}